\documentclass[11pt]{amsart}
\usepackage{amscd}      
\usepackage{amssymb}
\usepackage{amsmath}
\usepackage{xypic}      
\LaTeXdiagrams          
\usepackage[all,v2]{xy}
\usepackage{dsfont}
\xyoption{2cell} \UseAllTwocells \xyoption{frame} \CompileMatrices
\allowdisplaybreaks[3]
\usepackage{hyperref}
\usepackage{xcolor}

\usepackage{latexsym}
\usepackage{epsfig}
\usepackage{amsfonts}
\usepackage{enumerate}

\newtheorem{prop}{Proposition}[section]
\newtheorem{lem}[prop]{Lemma}

\newtheorem{cor}[prop]{Corollary}
\newtheorem{thm}[prop]{Theorem}

\numberwithin{equation}{section}
\theoremstyle{definition}

\newtheorem{defn}[prop]{Definition}

\newtheorem{example}[prop]{Example}

\newtheorem{rmk}[prop]{Remark}

\newcommand{\mmod}{{\!/\!\!/\!\!/\!\!/}}

            {\nolinebreak $\Box$ \end{trivlist}}

\newcommand{\noprint}[1]{}

\newcommand{\virt}{\mbox{\tiny virt}}
\newcommand{\red}{\mbox{\tiny red}}

\newcommand{\pt}{\mathop{pt}}

\newcommand{\MM}{{\mathfrak M}}

\newcommand{\zz}{{\mathbb Z}}
\newcommand{\hh}{{\mathbb H}}

\newcommand{\aaa}{{\mathbb A}}

\newcommand{\qq}{{\mathbb Q}}
\newcommand{\pp}{{\mathbb P}}
\newcommand{\cc}{{\mathbb C}}
\newcommand{\rr}{{\mathbb R}}

\newcommand{\cC}{{\mathcal C}}

\newcommand{\oO}{{\mathcal O}}

\newcommand{\A}{{\mathcal A}}
\newcommand{\M}{\mathcal{M}(\mathbf{\mathcal{A}})}

\newcommand{\CR}{\mbox{\tiny CR}}
\newcommand{\bbox}{\mbox{\tiny Box}}
\newcommand{\bbbox}{\mbox{Box}}

\newcommand{\MHL}{\mathcal{M}_{\lambda}(\mathcal{H})}
\newcommand{\age}{\mbox{age}}
\newcommand{\cM}{\overline{\mathcal{M}}}
\newcommand{\NE}{\mbox{NE}}
\newcommand{\inv}{\mbox{inv}}
\newcommand{\bbT}{\mathbb{T}}
\newcommand{\lcm}{\mbox{lcm}}
\newcommand{\bbig}{\mbox{\tiny big}}

\newcommand{\rk}{\mathop{\rm rk}}

\newcommand{\ev}{\mathop{\rm ev}\nolimits}

\newcommand{\Aut}{\mathop{\rm Aut}\nolimits}

\numberwithin{equation}{section}

\def\Label#1{\label{#1}{\tt [#1]}\phantom{h}}
\def\Label{\label}

\title[Hypertoric geometry and GW theory]{Hypertoric geometry and Gromov-Witten theory}

\author{Yunfeng Jiang}
\address{Department of Mathematics\\ University of Kansas\\ 405 Snow Hall 1460 Jayhawk Blvd\\Lawrence\\ KS 66045\\ USA}
\email{y.jiang@ku.edu}
\author{Hsian-Hua Tseng}
\address{Department of Mathematics\\ Ohio State University\\ 100 Math Tower,  231 West 18th Ave. \\Columbus, \\OH 43210 \\ USA}
\email{hhtseng@math.ohio-state.edu}
\date{\today}

\begin{document}

\begin{abstract}
We study Gromov-Witten theory of hypertoric Deligne-Mumford stacks from two points of view. From the viewpoint of representation theory, we calculate the operator of small quantum product by a divisor, following \cite{BMO}, \cite{MO}, \cite{MS}. From the viewpoint of Lawrence toric geometry, we compare Gromov-Witten invariants of a hypertoric Deligne-Mumford stack with those of its associated Lawrence toric stack.
\end{abstract}

\maketitle

\tableofcontents

\section{Introduction}

\subsection{Background}
In recent years, Gromov-Witten theory of symplectic varieties has been proved to have  deep connections to geometric representation theory by the work of Braverman-Maulik-Okounkov \cite{BMO}, Maulik-Okounkov \cite{MO}.  In Braverman-Maulik-Okounkov \cite{BMO}, the authors prove a quantum product formula by a divisor class for a smooth symplectic variety 
$X$, which is close related to the symplectic resolution of singularities. 

Symplectic Deligne-Mumford stacks are the corresponding symplectic resolutions in the stacky world.  The Gromov-Witten theory of smooth  Deligne-Mumford stacks has been developed in algebraic category by Abramovich-Graber-Vistoli \cite{AGV} and in symplectic category by Chen-Ruan \cite{CR2}.  It is interesting to study the quantum cohomology of symplectic Deligne-Mumford stacks and explore the relationship to representation theory.  Hypertoric Deligne-Mumford stacks (hypertoric DM stack), defined in \cite{JT} using stacky hyperplane arrangements,  are hyperk\"ahler analogue of K\"ahler toric Deligne-Mumford stacks.  They provide important examples of smooth symplectic Deligne-Mumford stacks.  In this paper we study the quantum cohomology of hypertoric Deligne-Mumford stacks, and generalize the quantum product formula by a divisor class in \cite{BMO} to hypertoric cases.  Note that for the smooth hypertoric varieties, the quantum cohomology has been studied in \cite{MS}.

\subsection{Hypertoric Deligne-Mumford stacks}

Let $N$ be a finitely generated abelian group of rank $d$ and
$N\to \overline{N}$ the natural projection modulo
torsion. Let $\beta:
\mathbb{Z}^{m}\to N$ be a homomorphism determined by a
collection of nontorsion integral vectors $\{b_{1},\cdots,b_{m}\}\subseteq
N$.  The map $\beta$ is required to  have finite cokernel. The Gale dual  (in the sense of \cite{BCS}) of $\beta$ is denoted by $\beta^{\vee}:
(\mathbb{Z}^{m})^{*}\to DG(\beta)$. A $generic$ element 
$\theta$ in $DG(\beta)$ and the vectors $\{\overline{b}_{1},\cdots,\overline{b}_{m}\}$
determine a hyperplane arrangement $\mathcal{H}=(H_{1},\cdots,H_{m})$ in
$N_{\mathbb{R}}^{*}$. We call $\mathcal{A}:=(N,\beta,\theta)$
a {\em stacky hyperplane arrangement}. 

There is a stacky fan $\mathbf{\Sigma_\theta}$ associated to the stacky hyperplane arrangement $\A$, and the associated 
toric Deligne-Mumford stack $\chi(\mathbf{\Sigma_\theta})$ is called  the {\em Lawrence toric Deligne-Mumford stack}.
The stack $\chi(\mathbf{\Sigma_\theta})=[U/G]$ is a quotient stack, where $U\subset \cc^{2m}=T^*\cc^m$ is an open subvariety determined by the irrelevant ideal of the Lawrence fan $\Sigma_{\theta}$,  $G$ is a finitely generated abelian group,  and the $G$-action on $U$ is determined by the stacky fan $\mathbf{\Sigma_\theta}$.  The hypertoric Deligne-Mumford stack $X_{\A}$ is defined as the quotient stack
$[Y/G]$,  where $Y\subset U$ is a  closed subvariety determined by a prime ideal in $T^*\cc^m$, see \S \ref{hyperplane}.  Hence the hypertoric 
Deligne-Mumford stack $X_{\A}$ is a closed substack of the Lawrence toric Deligne-Mumford stack $\chi(\mathbf{\Sigma_\theta})$. 
The topology of $X_\A$ is determined by the hyperplane arrangement $\A$.  In $\mathcal{H}$, the hyperplanes $H_i$ bound finite number of polytopes whose corresponding toric Deligne-Mumford stacks in the sense of \cite{BCS}, \cite{Jiang}, form the core $C(X_\A)$ of  $X_\A$.   The core $C(X_\A)$ is the deformation retract of $X_\A$. 

The map $\beta $  determines a multi-fan $\Delta_{\mathbf{\beta}}$ in $N_{\rr}$, which consists of  cones generated by linearly independent subsets $\{\overline{b}_{i_{1}},\cdots,\overline{b}_{i_{k}}\}$ in $\overline{N}$ for  $\{i_{1},\cdots,i_{k}\}\subset\{1,\cdots,m\}$, see \S \ref{substack}. We define a set $\mbox{Box}(\Delta_{\mathbf{\beta}})$ consisting of all  pairs $(v,\sigma)$, where $\sigma$ is a cone in the multi-fan $\Delta_{\mathbf{\beta}}$,  $v\in N$ such that  $\overline{v}=\sum_{\rho_{i}\subset\sigma}\alpha_{i}\overline{b}_{i}$ for 
$0<\alpha_{i}<1$.  For $(v,\sigma)\in \mbox{Box}(\Delta_{\mathbf{\beta}})$ we consider a closed substack of $X_\A$ given by the quotient stacky hyperplane arrangement
$\mathcal{A}(\sigma)$. The inertia stack of  $X_\A$ is the disjoint union of all such closed substacks, see \S \ref{substack}.  The Chen-Ruan cohomology $H^*_{\CR}(X_\A)=H^*(IX_\A)$ is the same as the cohomology of inertia stack up to the Chen-Ruan degree shifting.  The stack $X_\A$ admits a torus $\bbT:=(\cc^\times)^m\times\cc^\times$ action, where the $(\cc^\times)^m$-action is induced from the standard action on $T^*\cc^m$ and the extra $\cc^\times$ acts by scaling the fibre. We consider the $\bbT$-equivariant Chen-Ruan cohomology  $H^*_{\CR,\bbT}(X_\A)$ of $X_\A$.  Similar to the main result in \cite{JT}, the $\bbT$-equivariant Chen-Ruan cohomology  $H^*_{\CR,\bbT}(X_\A)$ is described by the matroid complex of the multi fan $\Delta_\beta$, see \S \ref{Equivariant:Chen-Ruan:cohomology}.

\subsection{Symplectic resolution and the Steinberg correspondence}\label{subsection:symplectic:resolution}

The genericness of $\theta\in DG(\beta)$ implies that the hyperplane arrangement $\mathcal{H}$ is simple. Hence the hypertoric Deligne-Mumford stack $X_{\A}$ is smooth.   There is a  nowhere vanishing symplectic form $\omega$ on $X_\A$ induced from the standard symplectic form on $T^*\cc^m$, thus $X_A$ is a  smooth symplectic Deligne-Mumford stack. 

If we choose $\theta=0$, then the hyperplanes $H_i$ in $\mathcal{H}$ all pass through the origin.  We denote the corresponding 
hypertoric stack by $X^0$.  It is not Deligne-Mumford, and is a singular stack.  Let $\overline{X}^0$ be its good moduli space in the sense of Alper \cite{Alper}.  There is a contraction map 
$$\phi: X_{\A}\to \overline{X}^0$$
which contracts the core $C(X_\A)$ to singular points. 
Hence the hypertoric Deligne-Mumford stack $X_\A$ is a symplectic resolution of the singular symplectic  variety
$\overline{X}^0$. 

Let $X_\A\times_{\overline{X}^0}X_\A$ be the fibre product.  The components whose dimension are the same as $X_\A$ are called the {\em Steinberg stack}.  Let $\mathbf{Z}$ be the union of all such components and $I\mathbf{Z}$ the inertia stack of 
$\mathbf{Z}$.  Then  $I\mathbf{Z}$ gives a correspondence on the $\bbT$-equivariant Chen-Ruan cohomology of $X_\A$ as follows. 
First the fibre product $X_\A\times_{\overline{X}^0}X_\A\subset X_\A\times X_\A$ is inside the product of $X_\A$. 
By Poincar\'e  duality, the
 inertia stack $[I\mathbf{Z}]$, taken as a cycle in $H_*(I(X_\A\times X_\A))$,  yields a class in the cohomology $H^*_{\bbT}(I(X_\A\times X_\A))$.  
Let 
$$\inv: IX_\A\to IX_\A$$
be the involution sending $(x,g)\mapsto (x,g^{-1})$ for 
$x\in X_\A$, $g\in \Aut(x)$.  
The Steinberg correspondence
\begin{equation}\label{steinberg:correspondence:introduction}
I\mathbf{Z}: H^*_{\CR,\bbT}(X_\A)\to H^*_{\CR,\bbT}(X_\A)
\end{equation}
is given by:
$$I\mathbf{Z}(\alpha)=\inv^\star Ip_{2,\star}(Ip_1^\star\alpha\cup [I\mathbf{Z}]),$$
where $p_i: X_\A\times X_\A\to X_\A$ are the projections for $i=1,2$ and $Ip_i: IX_\A\times IX_\A\to IX_\A$ are the projections on the corresponding inertia stacks. 
The Steinberg correspondence gives rise to an endomorphism on the $\bbT$-equivariant Chen-Ruan cohomology of $X_\A$.
We prove that the equivariant  quantum product by divisor classes of $X_\A$ is given by Steinberg correspondence in (\ref{steinberg:correspondence:introduction}).

%

\subsection{Small quantum product}

Let $\NE(X_{\A})\subset H_2(X_{\A},\rr)$ be the cone generated by classes of effective curves, and 
$\NE(X_{\A})_{\zz}=\{d\in H_2(X_{\A},\zz): d\in\NE(X_{\A})\}$.  We denote $R_{\bbT}:=H_{\bbT}^{*}(\pt)$ which is the $\bbT$-equivariant cohomology of a point, and let 
$$R_{\bbT}[\![Q]\!] = \left\{ \sum_{d \in \tiny\NE(X)_\zz} a_d Q^d : a_d \in R \right\}.$$ 
Here  $Q$ is a so-called \emph{Novikov variable}~\cite[III~5.2.1]{Manin}.  For $\gamma_i, \gamma_j\in H_{\CR,\bbT}^*(X_\A)$, the small equivariant quantum product is defined by:
\begin{equation}\label{quantum:product}
(\gamma_i\star \gamma_j, \gamma_k)=\sum_{d\in \tiny\NE(X)_{\zz}}Q^d\langle \gamma_i, \gamma_j, \gamma_k\rangle_{0,3,d}^{X_\A}.
\end{equation}

Here $\langle -,-,-,\rangle_{0,3,d}^{X_\A}$ are $3$-point genus $0$ $\bbT$-equivariant Gromov-Witten invariants of $X_\A$. Genus $0$ $\bbT$-equivariant Gromov-Witten invariants of $X_\A$ are defined as integrations against the $\bbT$-equivariant virtual fundamental class $[\cM_{0,n}(X_\A,d)]^{\virt}$ of the moduli stack $\cM_{0,n}(X_\A,d)$ of twisted stable maps from genus zero twisted curve with $n$-marked points to $X_\A$ of degree $d\in H_2(X_\A,\zz)$. The existence of a holomorphic symplectic form on $X_\A$ implies that the obstruction sheaf of $\cM_{0,n}(X_\A,d)$ has a trivial quotient, i.e. a cosection in the sense of \cite{KL}. This yields a {\em reduced virtual fundamental class} $[\cM_{0,n}(X_\A,d)]^{\red}$ for $\cM_{0,n}(X_\A,d)$. This reduced class satisfies 
$$[\cM_{0,n}(X_\A,d)]^{\virt}=\hbar\cdot [\cM_{0,n}(X_\A,d)]^{\red}.$$
Here $\hbar$ is the equivariant parameter for the $\cc^\times\subset \bbT=T\times \cc^\times$ action. See \S \ref{reduced:virtual:cycle} for more details. 

Let $u_i$ be a divisor class for $X_{\A}$. The divisor equation of Gromov-Witten invariants reduces the determination of quantum product by $u_i$ to two-point invariants of $X_{\A}$.  By  the relation between virtual fundamental cycle and reduced virtual cycle, we are reduced to calculate the pushforward by the evaluation map $\mathbf{ev}: \cM_{0,2}(X_\A, d)\to IX_{\A}\times IX_{\A}$:
$$\Gamma_2:=\mathbf{ev}_{\star}([\cM_{0,2}(X_\A, d)]^{\red})\subset H_*(IX_{\A}\times IX_{\A}).$$
Components of $\Gamma_2$ all have Chen-Ruan degree $\dim(X_\A)$ in $H^*_{\CR}(X_\A\times X_\A)$, which we call {\em Lagrangian cycle} in orbifold sense. 

Following \cite{MO}, such  cycle $\Gamma_2$ supports on $I\mathbf{Z}\subset I(X_\A\times X_\A)$, and  is called the {\em Steinberg Variety}. The correspondence 
\begin{equation}\label{Steinberg:correspondence:introducation2}
\Gamma_2: H_{\CR,\bbT}^*(X_{\A})\to H_{\CR,\bbT}^*(X_{\A})
\end{equation}
is given by:
$$\Gamma_2(\gamma)=\inv^\star Ip_{2,\star}(Ip_1^\star\gamma\cup \Gamma_2),$$
where $Ip_i: IX_\A\times IX_{\A}\to IX_{\A}$ is the projection for $i=1,2$,
and $\Gamma_2$ is taken as a cohomology class in $H_{\bbT}^*(IX_\A\times IX_\A)$. 
This is the Steinberg correspondence   
in (\ref{steinberg:correspondence:introduction}).
Recall that the index set of the components of 
$IX_\A$ is given by $I$.  We can write 
$$\Gamma_2:=\bigoplus_{\substack{(f_1, f_2):\\
 f_1, f_2\in I}}\Gamma_{f_1,f_2}.$$
 
A circuit $S\subset \A$ is a minimal subset $S\subset \{1,\cdots,m\}$ such that  $\{\overline{b}_{i}|i\in S\}$ are linearly dependent.  For any circuit $S\subset \A$, there is an associated curve class $\beta^S\in H_2(X_\A)$ defined as follows: there is a splitting $S=S^+\cup S^-$ and positive integers $w_i$ such that
 $$\beta^S:=\sum_{i\in S^+}w_i e_i-\sum_{i\in S^-}w_j e_j.$$
 Set $\mathbf{w}:=\{w_i|i\in S\}$.
 

The deformation of the hypertoric Deligne-Mumford stack $X_\A$ is classified by the image of the symplectic form in the second cohomology $H^2(X_\A,\cc)$, which is isomorphic to $DG(\beta)_{\cc}$.  In \S \ref{Section:deformation:hypertoricDM} we construct the following diagram:
\begin{equation}\label{deformation:diagram}
\xymatrix{
X_{\A} \ar[r]\ar[d] &~\widetilde{X}\ar[d]_{\phi}\\
0\ar[r]& DG(\beta)_{\cc}
}
\end{equation}
such that for {\em sub-regular parameters} $\lambda\in  DG(\beta)_{\cc}$ the deformation $X_\lambda:=\MHL$ has nice properties. The stack $X_\lambda$ contains a closed substack $\cM^S$ for each circuit $S\subset \A$, which is a weighted projective bundle over an affine base and the fibre normal bundle is the cotangent bundle of the weighted projective stack $\pp^{|S|-1}_{\mathbf{w}}$.  All curve classes of $X_\lambda$ lie in $\cM^S$, see \S \ref{Section:deformation:hypertoricDM}. Deformation invariance of Gromov-Witten invariants implies that  Gromov-Witten invariants in class $\beta^S$ can be computed by Gromov-Witten invariants of $T^*\pp^{|S|-1}_{\mathbf{w}}$.

Thus components of $IX_\lambda$ are all classified by the set $F=\{\frac{a_i}{w_i}| 0\leq a_i\leq w_i\}$.  Let  $l(\mathbf{w}):=\lcm(w_i|i\in S)$.  For any $f_1=\frac{a_1}{w_i}$ and  $f_2=\frac{a_2}{w_j}$, let  $\gamma(f_1,f_2)\in \{0,1,\cdots,l(\mathbf{w})-1\}$ such that \footnote{The existence of $\gamma(f_1, f_2)$ imposes constraints on the possible pairs $(f_1, f_2)$.}  $\langle\frac{\gamma(f_1,f_2)}{w_i}\rangle=f_1$, $\langle\frac{\gamma(f_1,f_2)}{w_j}\rangle=f_2$, and $w_k|\gamma(f_1,f_2)$ for $w_k\neq w_i, w_j$. 

\begin{thm}\label{Quantum:by:a:divisor}
Let $u_i$ be a divisor class of $X_\A$. Then the small equivariant quantum product  formula  by the divisor $u_i$ is given by:
$$u_i\star -=u_i\cdot -+\sum_{\substack{S\subset \A: \\
\tiny\mbox{circuit}}}
\hbar\cdot \langle u_i, \beta^{S}\rangle\cdot (-1)^d\cdot\sum_{\substack{(f_1,\sigma), (f_2,\tau)\in \bbox(\Delta_{\beta})\\}}
\frac{(Q^S)^{\gamma(f_1,f_2)+l(\mathbf{w})\cdot\delta_{\gamma(f_1,f_2),0}}}
{1-(Q^S)^{l(\mathbf{w})}}\Gamma_{f_1, f_2}(-),$$
where 
$\Gamma_{f_1, f_2}\in H^{2n}_{\CR,\bbT}(X_{\A}\times X_{\A})$ is the Steinberg correspondence in (\ref{Steinberg:correspondence:introducation2}).
\end{thm}

Theorem \ref{Quantum:by:a:divisor} is proved by a detail analysis of two point Gromov-Witten invariants of the hypertoric Deligne-Mumford stack $X_\A$.  Under evaluation maps, the image of the reduced virtual fundamental cycle $[\cM_{0,n}(X_\A,d)]^{\red}$ will have orbifold diegree $\dim(X_\A)$ inside $I(X_\A\times X_\A)$, hence the calculation of two point Gromov-Witten invariants is reduced to Steinberg correspondence in (\ref{Steinberg:correspondence:introducation2}), see \S \ref{Section:proof:main:result} for more details.  

Although the virtual fundamental cycle  $[\cM_{0,n}(X_\A,d)]^{\red}$ is a routine generalization from the case when $X_\A$ is a  smooth variety, the calculation of two point Gromov-Witten invariants is given by torus localization in orbifold Gromov-Witten theory, which seems to be a new technique in this setting to calculate the small quantum product by a divisor class.

\subsection{Lawrence toric geometry}
We use the geometry of Lawrence toric Deligne-Mumford stacks to give a formula for the small quantum cohomology of $X_\A$. 

Theorem \ref{Quantum:by:a:divisor} determines the small equivariant quantum products by any divisor class $u_i\in H^2(X_{\A})$.  In order to have a formula for the small  quantum cohomology ring of $X_\A$, we prove that the small quantum ring of $X_\A$ is isomorphic to the small quantum ring of the corresponding Lawrence toric Deligne-Mumford stack $X_{\mathbf{\theta}}:=\chi(\mathbf{\Sigma_{\theta}})$, see Corollary \ref{quantum:ring:isomorphism}.  This is derived from a more general result that equates  the decendant Gromov-Witten invariants of $X_\theta$  with the decendant Gromov-Witten invariants of $X_\A$, see Proposition \ref{invariants:compare}.

Recall that in the Lawrence toric fan 
$\Sigma_{\theta}$, the lattice is given by 
$N_{L}$, and 
$$\beta_{L}: \zz^{2m}\to N_{L}$$
is given by integral vectors $\{b_{L,1},\cdots,
b_{L,m},b'_{L,1},\cdots,
b^{'}_{L,m}\}\subset N_{L}$, see \S \ref{LawrenceDMstack}.
The map $\beta_{L}$ is called the Lawrence lifting of $\beta: \zz^m\to N$. 
Let 
$$R_{\bbT}[[Q]][N_{\mathbf{\Sigma_{\theta}}}]^{\widehat{}}$$
be the ring generated over $R_\bbT[[Q]]$ by symbols $\{y^c| c\in N_L\}$, with the following multiplication: for $c_1, c_2\in N_L$, $$y^{c_1}\star y^{c_2}=Q^{l(c_1,c_2)}y^{c_1+c_2},$$
where $l(c_1, c_2)$ is defined as follows. Suppose that $\sigma_1, \sigma_2, \sigma$ are cones in $\Sigma_\theta$ such that $c_1\in \sigma_1, c_2\in \sigma_2, c_1+c_2\in \sigma$. We can write $c_1=\sum_i (c_{1i}b_{L,i}+c_{1i}' b_{L,i}')$,  $c_2=\sum_i (c_{2i}b_{L,i}+c_{2i}' b_{L,i}')$,  $c_1+c_2=\sum_i (c_{12i} b_{L,i}+c_{12i}' b_{L,i}')$, where it is understood that $c_{1i}=0$ if $b_{L,i}\notin \sigma_1$, $c_{1i}'=0$ if $b_{L, i}'\notin \sigma_1$, and likewise for $c_{2i}, c_{2i}', c_{12i}, c_{12i}'$. Put $$l(c_1,c_2):=\sum_j ((c_{1j}+c_{2j}-c_{12j})e_j+(c_{1j}'+c_{2j}'-c_{12j}')e'_j)\in \mathbb{Q}^m\oplus \mathbb{Q}^m.$$ 
Note that $l(c_1,c_2)=0$ if $c_1, c_2$ belong to the same cone.

\begin{thm}\label{main2:small:quantum:ring}
Let $X_\A$ be the hypertoric Deligne-Mumford stack associated to the stacky hyperplane arrangement $\A$. Then the equivariant small quantum cohomology of $X_\A$ is $$QH^*_{\bbT}(X_\A)\cong \frac{R_{\bbT}[[Q]][N_{\mathbf{\Sigma_{\theta}}}]^{\widehat{}}}{\langle y^{b_{L, i}}+y^{b'_{L,i}}-\hbar | i=1,...,m\rangle}.$$
\end{thm}

Theorem \ref{main2:small:quantum:ring} is obtained by calculating the small quantum cohomology ring of the Lawrence toric Deligne-Mumford stack $X_{\mathbf{\theta}}$. The presentation above is obtained from calculations with the (extended) $I$-function of $X_{\mathbf{\theta}}$. The isomorphism follows from toric mirror theorem \cite{CCIT} and calculations of the mirror map along $H^2$, see \S \ref{sec:pf_ring_presentation}.

\begin{rmk}
\hfill
\begin{enumerate}
\item
Theorem \ref{main2:small:quantum:ring} specializes to the calculation in \cite{MS}.  If the hypertoric Deligne-Mumford stack $X_\A$ is a smooth variety, which means that the hyperplane arrangement $\mathcal{H}$ is unimodular, see \cite{HS}, then the corresponding Lawrence toric variety 
$X_\theta$ is a smooth variety. 
There is an one-to-one correspondence between the generators $d\in H^2(X_\theta,\zz)\cong DG(\beta)_\zz$ and the circuits $S\subset \A$.  The splitting $S=S^+\cup S^{-}$ is actually given by 
$S^+=\{i\in S| \langle D_i,d\rangle>0\}$ and $S^-=\{i\in S| \langle D_i,d\rangle\}<0)$. 
By Poincar\'e  duality, such a generator $d$ determines a curve class 
$\beta^{S}\in H_2(X_\theta,\zz)$.
The quantum Stanley-Reisner ideal $QSR$ in \cite[Theorem 1.1]{MS} can be obtained directly from Theorem \ref{main2:small:quantum:ring}. 
\item
The presentation in Theorem \ref{main2:small:quantum:ring} may be rewritten in the form of generators and relations, along the line of \cite[Theorem 4.9]{TW}.
\end{enumerate}
\end{rmk}

\subsection{Further studies}

The monodromy conjecture for symplectic resolutions was formulated by Braverman-Maulik-Okounkov in \cite{BMO}. 
Roughly speaking a compactified K\"ahler moduli space $\mathcal{M}$ have large radius points $0, \infty$ such that the corresponding two symplectic Deligne-Mumford stacks $X_1, X_2$ are birational equivalent.  The derived categories of them are expected to be equivalent:
$$D^b(X_1)\cong D^b(X_2).$$
The equivalence is given by a choice of path from $0$ to $\infty$ and thus giving a map 
$$\rho: \pi_{1}(\mathcal{M})\to \mbox{Aut}(D^b(X_i)).$$  
Moreover they expect that this map is considered in the level of K-theory
$$\rho: \pi_{1}(\mathcal{M})\to \mbox{Aut}(K_0(X_i)).$$ 
The monodromy conjecture says that the mondromy of the quantum connection 
$\bigtriangledown$  for $X_i$ is the same as the above mondromy given by the equivalence on the K-theory. 

Note that this  conjecture is already known in the case of crepant birational transformation of toric Deligne-Mumford stacks. 
In \cite{CIJ}, the crepant transformation conjecture in Gromov-Witten theory was proved for a crepant birational transformation of toric Deligne-Mumford stacks given by single wall crossing.   Let $X_+\dasharrow X_-$ be a crepant birational map between two smooth toric Deligne-Mumford stacks.  They are derived equivalent, which is given by Fourier-Mukai transform, see \cite{Kawamata} and \cite{CIJS}.   In \cite{CIJ}, the authors prove that the equivariant Fourier-Mukai transformation on the K-theory matches the analytic continuation of the  $I$-function, hence matches the quantum connection which is determined by the $I$-function.  Fourier-Mukai transformation depends on a choice of path in the mirror of the toric Deligne-Mumford stacks, and applying twice of Fourier-Mukai gives the monodromy.  The matching of the Fourier-Mukai transform with analytic continuation of quantum connections implies that their corresponding monodromies are the same. 

In \cite{JT3}, we will study the case of wall crossing of hypertoric Deligne-Mumford stacks by varying the stability parameters 
$\theta$.  The wall crossing of hypertoric Deligne-Mumford stacks is actually given by a single wall crossing of Lawrence toric Deligne-Mumford stacks studied in \cite{CIJ}.  Hence  the wall crossing is given by the Mukai type flops of hypertoric
Deligne-Mumford stacks.  Several authors, see \cite{Cautis}, \cite{Cautis2}, \cite{Kawamata}, already proved that their derived categories are equivalent, and the kernel is also given by Fourier-Mukai type transform.   We expect to prove that the Fourier-Mukai transform matches the analytic continuation of quantum connections of the hypertoric Deligne-Mumford stack.

The Monodromy conjecture works for any  two crepant  birational transformation of   symplectic Deligne-Mumford stacks.  One type of such birational equivalence is the Mukai type  flops, which are studied by many mathematician, see for instance \cite{Cautis}, \cite{Cautis2}, \cite{Kawamata}.   In some nice situation, their derived categories are equivalent and Fourier-Mukai type  transformation gives the equivalence.  These are more general cases than hypertoric Deligne-Mumford stacks. 
We hope that our approach in this project may shed light on proving the conjecture  in more general cases.

\subsection{Outline}
The rest of this paper is organized as follows.  The notion of hypertoric Deligne-Mumford stacks and their properties are reviewed in \S \ref{hyperplane}. 
In \S \ref{Equivariant:Chen-Ruan:cohomology} we determine the equivariant Chen-Ruan cohomology of hypertoric Deligne-Mumford stacks. Gromov-Witten theory and reduced virtual fundamental cycles are reviewed in \S \ref{GW:theory}. 
In \S \ref{Section:deformation:hypertoricDM} we discuss the deformation of hypertoric Deligne-Mumford stacks by sub-regular parameter under the moment maps.  We discuss the Steinberg correspondence in \S \ref{Section:Steinberg:correspondence} for symplectic resolution of  hypertoric Deligne-Mumford stacks.   We prove Theorem \ref{Quantum:by:a:divisor}  in \S \ref{Section:proof:main:result}; and in 
\S \ref{Section:GW:theory:hypertoricDM} we study Gromov-Witten invariants of hypertoric Deligne-Mumford stacks. We  give a ring structure for the small  equivariant quantum cohomology of hypertoric Deligne-Mumford stacks and calculate two examples. 

\subsection{Set-up}
We work over the field of complex numbers. Cohomology groups are taken with rational coefficients. 

$N$ is a finitely generated abelian group of rank $d$. 

$N\to \overline{N}:=N/N_{\tiny\mbox{tor}}$ is the natural quotient map.

$\beta: \mathbb{Z}^{m}\to N$ is a group homomorphism determined by sending the standard basis of $\mathbb{Z}^m$ to a collection of nontorsion integral vectors $\{b_{1},\cdots,b_{m}\}\subseteq N$.

$\beta^{\vee}: \mathbb{Z}^{m}\to DG(\beta)$ is the Gale dual of $\beta$ as constructed in \cite{BCS}.
 
 For cones $\sigma_{1},\sigma_{2}$ in $\mathbb{R}^{d}$,
we use $\sigma_{1}\cup\sigma_{2}$ to represent the set of union of the 
generators of $\sigma_{1}$ and $\sigma_{2}$.
For a positive integer $m$, we use $[m]$ to represent the set 
$\{1,\cdots,m\}$.

For a rational number $c$, let $\langle c\rangle$ denote the fractional part of $c$, $\lceil c\rceil$ the ceiling of $c$, and $\lfloor c\rfloor$ the floor of $c$.

For a stacky hyperplane arrangement $\A$, we denote by 
$X_\A$ the associated hypertoric Deligne-Mumford stack throughout the paper.
For a sub-regular parameter $\lambda$ under the moment map, we denote by $X_\lambda:=\MHL$ the deformation of $X_\A$ by this sub-regular parameter $\lambda$. 
We set $\mbox{Box}:=\mbox{Box}(\Delta_{\beta})$, the box element of the multi-fan $\Delta_{\beta}$.

For a smooth Deligne-Mumford stack $X$ with a torus $\bbT$-action, we use $H_{\CR, \bbT}^*(X)$ to represent the equivariant Chen-Ruan cohomology of $X$,  $QH^*_{\bbT}(X)$ the equivariant small quantum cohomology ring, and 
$QH^*_{\bbT, \bbig}(X)$ the equivariant big quantum cohomology ring. The Chen-Ruan orbifold cup product is denoted by $\cdot$, the small quantum product is denoted by $\star$, and the big quantum product is denoted by $\star_{\bbig}$. 

\subsection{Acknowledgments}

We thank D. Edidin, N. Proudfoot and M. MacBreen for the discussions on symplectic resolution and K-theory of hypertoric Deligne-Mumford stacks.  Y. J. especially thanks Gufang Zhao to draw his attention to the MIT-Northeastern seminar series on quantum cohomology, geometric representation theory and monodromy conjecture.  Both authors are partially supported by Simons Foundation Collaboration Grants. 

\section{Preliminaries} 
\subsection{Hypertoric geometry}\label{hyperplane}
We recall the definition of hypertoric Deligne-Mumford stacks in sense of \cite{JT}.
\subsubsection{Stacky hyperplane arrangements}\label{Section:stacky:hyperplane:arrangement}
We introduce stacky hyperplane arrangements and explain their relations to extended stacky fans. 

Let $N$ be a finitely generated abelian group of rank $d$ and  $\beta:\mathbb{Z}^m\to N$  a map given by nontorsion
integral vectors $\{b_1, \cdots,b_m\}$. We have the following exact
sequences:
\begin{equation}\Label{exact1}
0\longrightarrow DG(\beta)^{*}\stackrel{(\beta^{\vee})^{*}}{\longrightarrow}
\mathbb{Z}^{m}\stackrel{\beta}{\longrightarrow} N\longrightarrow
Coker(\beta)\longrightarrow 0,
\end{equation}
\begin{equation}\Label{exact2}
0\longrightarrow N^{*}\longrightarrow
\mathbb{Z}^{m}\stackrel{\beta^{\vee}}{\longrightarrow}
DG(\beta)\longrightarrow Coker(\beta^{\vee})\longrightarrow 0,
\end{equation}
where $\beta^{\vee}$ is the Gale dual of $\beta$ (see \cite{BCS}). 
The map $\beta^{\vee}$ is given by the integral vectors $\{a_1,\cdots,a_m\}\subseteq DG(\beta)$. 
Choose a generic element $\theta\in DG(\beta)$ which lies in the image of $\beta^{\vee}$ and let
$\psi:=(r_{1},\cdots,r_{m})$ be a lifting of $\theta$ in
$\mathbb{Z}^{m}$ such that $\theta=-\beta^{\vee}\psi$. Note that $\theta$ is
generic if and only if it is not in any hyperplane of the
configuration determined by $\beta^{\vee}$ in $DG(\beta)_{\mathbb{R}}$. Let 
$M=N^{*}$ be the dual of $N$. $M_{\mathbb{R}}=M\otimes_{\mathbb{Z}}\mathbb{R}$ is a $d$-dimensional $\mathbb{R}$-vector space. 
Associated to $\theta$ there is a hyperplane arrangement
$\mathcal{H}=\{H_{1},\cdots,H_{m}\}$ in $M_{\mathbb{R}}$ defined by
$H_{i}$  the hyperplane
\begin{equation}\Label{arrangement}
H_{i}:=\{v\in M_{\mathbb{R}}|<b_{i},v>+r_{i}=0\}\subset M_{\mathbb{R}}.
\end{equation}
So (\ref{arrangement}) determines hyperplane arrangements in $M_{\mathbb{R}}$, up to translation
induced by the choice of the lifting $\psi:=(r_{1},\cdots,r_{m})$.

\begin{defn}
We call
$\mathcal{A}:=(N,\beta,\theta)$ a {\em stacky hyperplane arrangement}.
\end{defn}

It is
well-known that hyperplane arrangements determine the topology of
hypertoric varieties \cite{BD}. 
Let
$$\mathbf{\Gamma}=\bigcap_{i=1}^{m}F_{i}, \text{ where }F_{i}=\{v\in M_{\mathbb{R}}|<b_{i},v>+r_{i}\geq 0\}.$$ 
Let $\Sigma$ be the normal fan of 
$\mathbf{\Gamma}$ in $M_{\mathbb{R}}=\mathbb{R}^{d}$ with one
dimensional rays generated by
$\overline{b}_{1},\cdots,\overline{b}_{n}$. By reordering, we may
assume that $H_{1},\cdots,H_{n}$ are the hyperplanes that bound
the polytope $\mathbf{\Gamma}$, and $H_{n+1},\cdots,H_{m}$ are the
other hyperplanes. Then we have an extended stacky fan
$\mathbf{\Sigma}=(N,\Sigma,\beta)$ in the sense of \cite{Jiang}, where
$\beta:
\mathbb{Z}^{m}	\to N$ is given by
$\{b_{1},\cdots,b_{n},b_{n+1},\cdots,b_{m}\}\subset N$, and  $\{b_{n+1},\cdots,b_{m}\}$
are the extra data. 

By \cite{Jiang}, the extended stacky fan $\mathbf{\Sigma}$
determines a toric Deligne-Mumford stack
$\mathcal{X}(\mathbf{\Sigma})$. Its coarse moduli space is the
toric variety corresponding to the normal fan of
$\mathbf{\Gamma}$. According to \cite{BD}, a hyperplane arrangement
$\mathcal{H}$ is {\em simple} if the codimension of the nonempty
intersection of  any $l$ hyperplanes is $l$. A hypertoric variety
is the coarse moduli space of an  {\em orbifold} if the corresponding hyperplane arrangement
is simple.

\begin{rmk}
Consider the map $\mathbb{Z}^n\to N$ given by $\{b_1,...,b_n\}$. Then $(N,\Sigma,\mathbb{Z}^n\to N)$ is a stacky fan in the sense of  \cite{BCS}. The associated toric Deligne-Mumford stack is isomorphic to  $\mathcal{X}(\mathbf{\Sigma})$. 
\end{rmk}

\subsubsection{Lawrence toric Deligne-Mumford stacks}\label{LawrenceDMstack}
Consider the Gale dual map
$\beta^{\vee}: \mathbb{Z}^{m}\to
DG(\beta)$ in (\ref{exact2}).  We denote 
the Gale dual map of 
$$ \mathbb{Z}^{m}\oplus \mathbb{Z}^{m}\stackrel{(\beta^{\vee},-\beta^{\vee})}{\longrightarrow}
DG(\beta)$$ 
by 
\begin{equation}\Label{betal}
\beta_{L}: \mathbb{Z}^{2m}\rightarrow N_{L},
\end{equation} \\
where $\overline{N}_{L}$
is a lattice of dimension $2m-(m-d)$. The map $\beta_{L}$ is given by the integral vectors  $\{b_{L,1},\cdots,
b_{L,m},b'_{L,1},\cdots,
b^{'}_{L,m}\}$ and $\beta_{L}$ is called  the Lawrence lifting of $\beta$.

\begin{rmk}\label{lawrencermk-s}
By \cite[Remark 2.3]{JT},  the lattice $N_L=N\oplus \zz^m$ and 
$\{b_{L,1},\cdots,
b_{L,m},b'_{L,1},\cdots,
b^{'}_{L,m}\}$ are the vectors $$\{(b_{1},e_{1}),\cdots,
(b_{m},e_{m}),(0,e_{1}),\cdots,
(0,e_{m})\},$$ where $\{e_{i}\}$ are the standard bases of $\mathbb{Z}^{m}$.
\end{rmk}

Given the generic element $\theta$, let $\overline{\theta}$
be the natural image of $\theta$ under the projection $DG(\beta)\rightarrow \overline{DG(\beta)}$. 
Then the map $\overline{\beta}^{\vee}: \mathbb{Z}^{m}\rightarrow \overline{DG(\beta)}$ is given by $\overline{\beta}^{\vee}=(\overline{a}_{1},\cdots,\overline{a}_{m})$. For any basis of $\overline{DG(\beta)}$
of the form $C=\{\overline{a}_{i_{1}},\cdots,\overline{a}_{i_{m-d}}\}$, there exist unique
$\lambda_{1},\cdots,\lambda_{m-d}$ such that
$$\overline{a}_{i_{1}}\lambda_{1}+\cdots+\overline{a}_{i_{m-d}}\lambda_{m-d}=\overline{\theta}.$$
Let $\mathbb{C}[z_{1},\cdots,z_{m},w_{1},\cdots,w_{m}]$ be the coordinate ring of $\mathbb{C}^{2m}$. Let
$$\sigma(C,\theta)=\{\overline{b}_{L,i_{j}}~|\lambda_{j}>0\}\sqcup\{\overline{b}'_{L,i_{j}}|~\lambda_{j}<0\} \quad \text{and} \quad 
C(\theta)=\{z_{i_{j}}~|\lambda_{j}>0\}\sqcup\{w_{i_{j}}|~\lambda_{j}<0\}.$$
We put
\begin{equation}\Label{irrelevant}
\mathbf{\mathcal{I}}_{\theta}:=\left\langle\prod
C(\theta)|~C~\text{is a basis of}~\overline{DG(\beta)}\right\rangle,
\end{equation}
and 
\begin{equation}\Label{fan}
\Sigma_{\theta}:=\{\overline{\sigma}(C,\theta):~C~\text{is a basis of}~\overline{DG(\beta)}\},
\end{equation}
where $\overline{\sigma}(C,\theta)=\{\overline{b}_{L,1},\cdots,\overline{b}_{L,m},\overline{b}'_{L,1},\cdots,\overline{b}'_{L,m}\}\setminus\sigma(C,\theta)$ 
is the complement of $\sigma(C,\theta)$ and corresponds to a maximal cone in $\Sigma_{\theta}$.
From \cite{HS}, $\Sigma_{\theta}$ is the fan of a Lawrence toric variety $X(\Sigma_{\theta})$ corresponding to
$\theta$ in the lattice $\overline{N}_{L}$, and $\mathcal{I}_{\theta}$ is the irrelevant ideal.  
The construction above establishes the following
\begin{prop}[\cite{JT}, Proposition 2.5]\Label{lawrencestacky}
A stacky hyperplane arrangement $\mathcal{A}=(N,\beta,\theta)$
also gives a stacky fan $\mathbf{\Sigma_{\theta}}=(N_{L},\Sigma_{\theta},\beta_{L})$
which is called a Lawrence stacky fan.  
\end{prop}

\begin{defn}[\cite{JT}, Definition 2.6]\label{lawrencetoricdmstack}
The toric Deligne-Mumford stack $\mathcal{X}(\mathbf{\Sigma_{\theta}})$ is called the Lawrence toric Deligne-Mumford  stack.
\end{defn}

For  the map $\beta_{L}^{\vee}:
\mathbb{Z}^{m}\oplus \mathbb{Z}^{m}\to DG(\beta)$ given by 
$(\beta^{\vee},-\beta^{\vee})$, there is an exact sequence
\begin{equation}\Label{exact3}
0\longrightarrow N_{L}^{*}\longrightarrow
\mathbb{Z}^{2m}\stackrel{\beta_{L}^{\vee}}{\longrightarrow}
DG(\beta)\longrightarrow Coker(\beta_{L}^{\vee})\longrightarrow
0.
\end{equation}
Applying $Hom_\mathbb{Z}(-,\mathbb{C}^\times)$ to (\ref{exact3}) yields
\begin{equation}\Label{exact5}
1\longrightarrow \mu\longrightarrow
G\stackrel{\alpha^{L}}{\longrightarrow}
(\mathbb{C}^{\times})^{2m}\longrightarrow
T_{L}\longrightarrow 1, \
\end{equation}
where
$\mu:=Hom_{\mathbb{Z}}(Coker(\beta_{L}^{\vee}),\mathbb{C}^{\times})$ and $T_{L}$ is the torus 
of dimension $m+d$. 
From \cite{BCS} and Proposition \ref{lawrencestacky}, the toric DM stack $\mathcal{X}(\mathbf{\Sigma_{\theta}})$ is
the quotient stack $[(\mathbb{C}^{2m}\setminus V(\mathcal{I}_{\theta}))/G]$, where $G$ acts 
on $\mathbb{C}^{2m}\setminus V(\mathcal{I}_{\theta})$ through the map $\alpha^{L}$ in (\ref{exact5}). 

\subsubsection{Hypertoric Deligne-Mumford stacks}\label{Section:hypertoricDMstack}
Define an ideal in $\mathbb{C}[z,w]$ by:
\begin{equation}\Label{ideal1}
I_{\beta^{\vee}}:=\left\langle\sum_{i=1}^{m}(\beta^{\vee})^{*}(x)_{i}z_{i}w_{i}|~ x\in DG(\beta)^{*}\right\rangle,
\end{equation}
where $(\beta^{\vee})^{*}$ is the map in (\ref{exact1}) and $(\beta^{\vee})^{*}(x)_{i}$ is the $i$-th component
of the vector $(\beta^{\vee})^{*}(x)$.

According to Section 6 in \cite{HS}, $I_{\beta^{\vee}}$ is a prime ideal. Let $Y$ be the closed  subvariety of
$\mathbb{C}^{2m}\setminus V(\mathcal{I}_{\theta})$ determined by the ideal (\ref{ideal1}).
Since 
$(\mathbb{C}^{\times})^{2m}$ acts on $Y$ naturally and the
group $G$ acts on $Y$ through the map $\alpha^{L}$, we have the quotient stack $[Y/G]$. 
Since $Y\subseteq\mathbb{C}^{2m}\setminus V(\mathcal{I}_{\theta})$ is a closed subvariety, the quotient stack $[Y/G]$ is a closed substack 
of $\mathcal{X}(\mathbf{\Sigma_{\theta}})$, and is Deligne-Mumford.

\begin{defn}[\cite{JT}, Definition 2.7]
The hypertoric Deligne-Mumford stack
$X_\A$ associated to the  stacky
hyperplane arrangement  $\mathcal{A}$ is defined to be the quotient stack
$[Y/G]$.
\end{defn}

\begin{example}
Let $N=\mathbb{Z}$, $\Sigma$ the fan of
projective line $\mathbb{P}^{1}$, and $\beta:
\mathbb{Z}^{2}\to N$  given by $\{b_{1}=(-2),
b_{2}=(1)\}$. Then the Gale dual $\beta^{\vee}:
\mathbb{Z}^{2}\to \mathbb{Z}$ is given by the
matrix $\left[
\begin{array}{ccc}
1,&2
\end{array}
\right]$. Choose a generic element $\theta=(1)$
in $\mathbb{Z}$ which determines the fan $\Sigma$.
The stacky hyperplane arrangement is $\mathcal{A}=(N,\beta,\theta)$, 
$G=(\mathbb{C}^{\times})^{1}$ and $Y$ is the
subvariety of
$\text{Spec}(\mathbb{C}[z_{1},z_{2},w_{1},w_{2}])$ determined
by the ideal $I_{\beta^{\vee}}=(z_{1}w_{1}+z_{2}w_{2})$.  Then the hypertoric Deligne-Mumford stack is the cotangent bundle $T^*_{\pp(1,2)}$. 
\end{example}

Each Deligne-Mumford stack has an underlying coarse moduli space. 
Consider again the map $\beta^{\vee}: \mathbb{Z}^{m}\rightarrow DG(\beta)$ in (\ref{exact2}),
which is given by the vectors $(a_1,\cdots,a_m)$. As in \S \ref{LawrenceDMstack}, let 
$\overline{\theta}$
be the natural image of $\theta$ under the projection $DG(\beta)\rightarrow \overline{DG(\beta)}$. 
Then the map $\overline{\beta}^{\vee}: \mathbb{Z}^{m}\rightarrow \overline{DG(\beta)}$ is given by $\overline{\beta}^{\vee}=(\overline{a}_{1},\cdots,\overline{a}_{m})$. 
From the map $\overline{\beta}^{\vee}$ we get the  simplicial fan $\Sigma_{\theta}$ in 
(\ref{fan}). By \cite{BCS}, the toric variety $X(\Sigma_{\theta})=(\mathbb{C}^{2m}-V(\mathcal{I}_{\theta}))\slash G$, is the coarse moduli space of 
the Lawrence toric Deligne-Mumford stack $\mathcal{X}(\mathbf{\Sigma_{\theta}})$. The toric variety  $X(\Sigma_{\theta})$
is semi-projective, but not projective. 
In \cite{HS},
from $\beta^{\vee}$ and $\theta$, the authors define the hypertoric variety 
$Y(\beta^{\vee},\theta)$ as the complete intersection of the toric variety 
$X(\Sigma_{\theta})$ by the ideal (\ref{ideal1}), which is the geometric quotient 
$Y\slash G$.

\begin{prop}[\cite{JT}, Proposition 2.8]
The coarse moduli space of $X_\A$ is
$Y(\beta^{\vee},\theta)$.
\end{prop}

\subsubsection{Cores}\label{core:hypertoricDMstack}

Recall that a hypertoric Deligne-Mumford stack $X_\A\to \overline{X}_0$ is a symplectic resolution, and the core 
$C(X_{\A})$ is the fibre over most singular points of $\overline{X}_0$, which is the deformation retract of $X_\A$. 

The core is a finite union of toric Deligne-Mumford stacks.   Let 
$U\subset [m]$ be a finite subset, and set
$$\mathcal{P}_{U}:=\{v\in M_{\rr}| \langle b_i,v\rangle+r_i\geq 0 ~\text{if}~ i\in U,  ~\text{and}~ \langle b_i,v\rangle+r_i\leq 0 ~\text{if}~ i\notin U\}.$$
Then $\mathcal{P}_{U}$ is the polytope cut out by the cooriented hyperplanes of $\mathcal{H}=\{H_1,\cdots,H_m\}$ after reversing the coorientations of the hyperplanes with indices in $U$.  Assume that $\mathcal{P}_U$ is bounded, we denote by 
$\Sigma_{U}\subset N_{\rr}$ the normal fan of $\mathcal{P}_U$. Then 
$\mathbf{\Sigma}_{U}=(N, \Sigma_{U}, \beta)$ is an extended stacky fan in sense of \cite{Jiang}, and let 
$\chi(\mathbf{\Sigma_{U}})$ be the corresponding toric Deligne-Mumford stack. 

\begin{prop}
The core $C(X_\A)$ of hypertoric Deligne-Mumford stack $X_\A$ is
$$C(X_\A)=\bigcup_{\mathcal{P}_U \text{~bounded~}}\chi(\mathbf{\Sigma_{U}}).$$
\end{prop}

\begin{example}\label{example1:core}
Let $(N, \beta,\theta)$ be a stacky hyperplane arrangementgiven by:
$$\mathbb{Z}^4\stackrel{\beta}{\longrightarrow}\zz^2,$$
and $\beta$ is given by 
$$
\begin{cases}
b_1=(1,0);\\
b_2=(0,-1);\\
b_3=(0,1);\\
b_4=(-1,-2).
\end{cases}
$$
The generic element $\theta=(1,1)\in DG(\beta)=\zz^2$.  
The normal fan of the bounded polytope $\Gamma$ is the toric fan of a Hirzebruch surface. 
The core of the hypertoric Deligne-Mumford stack is a union of Hirzebruch surface and a weighted projective plane $\pp(1,1,2)$.
\end{example}

\subsubsection{Closed substacks of hypertoric Deligne-Mumford stacks}\label{substack}

Let $\mathcal{A}=(N,\beta,\theta)$ be a 
stacky hyperplane arrangement.
Consider the map $\beta: \mathbb{Z}^{m}\rightarrow N$ given by $\{b_{1},\cdots,b_{m}\}$.
Let $\mbox{Cone}(\beta)$ be a partially
ordered finite set of cones generated by
$\overline{b}_{1},\cdots,\overline{b}_{m}$. The partial ordering is
defined by requiring that $\sigma\prec\tau$ if $\sigma$ is a face of $\tau$. We have the minimum element $\hat{0}$ which is the cone
consisting of the origin. Let $\mbox{Cone}(\overline{N})$ be the set of
all convex polyhedral cones in the lattice $\overline{N}$. Then we
have a map
$$C: \mbox{Cone}(\beta)\longrightarrow \mbox{Cone}(\overline{N}),$$ such
that for any $\sigma\in \mbox{Cone}(\beta)$, $C(\sigma)$ is the
cone in $\overline{N}$. Then $\Delta_{\mathbf{\beta}}:=(C,\mbox{Cone}(\beta))$ is a simplicial multi-fan in the sense of \cite{HM}. 

Let $\sigma\in\Delta_{\beta}$ be a cone.  According to \cite[Section 4]{JT}, the stacky hyperplane arrangement $\A=(N,\beta,\theta)$ induces a quotient stacky hyperplane arrangement $\A/\sigma=(N(\sigma), \beta(\sigma), \theta(\sigma))$, whose corresponding hypertoric Deligne-Mumford stack 
$X_{\A/\sigma}$ is a closed substack of $X_\A$.  More details can be found in \cite[Section 4]{JT}. 

\begin{example}\label{inertia:stacks}[Inertia stacks]
Let $N_{\sigma}$ be the sublattice generated by $\sigma$, and $N(\sigma):=N/N_{\sigma}$. Note that when $\sigma$ is a  top
dimensional cone, $N(\sigma)$ is the local orbifold group in the
local chart of the coarse moduli space of the hypertoric toric Deligne-Mumford
stack. Given the multi-fan $\Delta_{\mathbf{\beta}}$, we consider the pairs
$(v,\sigma)$, where $\sigma$ is a cone in $\Delta_{\mathbf{\beta}}$, 
$v\in N$ such that $\overline{v}=\sum_{\rho_{i}\subseteq \sigma}\alpha_{i}b_{i}$ for 
$0<\alpha_{i}<1$. Note that $\sigma$ is the minimal cone in $\Delta_{\mathbf{\beta}}$
satisfying the above condition. Let $\mbox{Box}:=\mbox{Box}(\Delta_{\mathbf{\beta}})$ be the set of all such pairs  
$(v,\sigma)$.

\cite[Proposition 4.7]{JT} determines the 
 inertia stack of $X_\A$, which  is given by
\begin{equation}\label{inertia}
IX_\A=\coprod_{(v,\sigma)\in
\bbox(\Delta_{\mathbf{\beta}})}X_{\A/\sigma}.
\end{equation}
\end{example}

\subsection{Equivariant Chen-Ruan cohomology}\label{Equivariant:Chen-Ruan:cohomology}

Let $X_\A$ be the hypertoric Deligne-Mumford stack associated to a stacky hyperplane arrangement $\A$.  From the construction in \S \ref{Section:hypertoricDMstack}, there is a torus $T:=(\cc^\times)^m$ action on $X_\A$.  From the exact sequence:
$$1\rightarrow \mu\longrightarrow G\longrightarrow (\cc^\times)^m\longrightarrow (\cc^\times)^d\rightarrow 1$$
the $T$-action on $X_\A$ induces a $(\cc^\times)^d$-action on $X_\A$. 

There is another $\cc^\times$-action on the fibre of 
$$T^*\cc^m=\cc^m\times (\cc^m)^*$$
by scaling.  We consider $\bbT:=T\times \cc^\times$-equivariant Chen-Ruan cohomology of $X_\A$. 

\subsubsection{$\bbT$-equivariant cohomology of $X_\A$.}
The torus $T$-equivariant cohomology of $X_\A$ is described by the multi-fan $\Delta_\beta$, i.e. the Matroid $M_\beta$. 
Each $b_i\in \A$ defines a line bundle 
$$L_i=[Y\times\cc/G]$$
where $G$ acts on the fibre by the $i$-th component 
$\alpha: G\to (\cc^\times)^m$. 
Let $\lambda_i$ be the parameters of the $T$-action on $X_\A$. 
Let $\{u_i | 1\leq i\leq m\}$ be the $T$-equivariant Chern class of $L_i$ over $X_\A$. 
If $D_i$ is the divisor corresponding to the line bundle $L_i$. Then 
$$u_i=D_i-\lambda_i.$$
The following result is due to 
Hausel and Sturmfel \cite{HS}:
\begin{prop}
Let $\A=(N,\beta,\theta)$ be a stacky hyperplane arrangement and $X_\A$ the corresponding hypertoric Deligne-Mumford stack. Then
$$H^*_{T}(X_\A)=R_\bbT [u_1,\cdots,u_m]/I_{M_\beta},$$
where 
$$I_{M_\beta}=\{u_{i_1}\cdots u_{i_k}| \overline{b}_{i_1}, \cdots, \overline{b}_{i_k}~\mbox{linearly dependent in~} \overline{N}\}.$$
\end{prop}

We consider the extra factor $\cc^\times$ action in the $\bbT$-equivariant cohoology of $X_\A$. The extra $\cc^\times$-action on 
$T^*\cc^m$ descends to a $\bbT$-equivariant line bundle over $X_\A$ with the first Chern class $\hbar$. 
Recall the line bundle $L_i$ over $X_\A$ just constructed before.   The line bundle $L_i$ can be thought as a divisor
$$\{(z_i,w_i)\in Y| z_i=0\}.$$
Now if we work $\bbT$-equivariantly, we will have the following $L_i^{-1}$ over $X_\A$ as:
$$\{(z_i,w_i)\in Y| w_i=0\}$$
and we have a $\cc^\times$-action on $L_i^{-1}$, so
$$c_1(L_i^{-1})=\hbar-u_i.$$


As in \cite{GH}, \cite{HP}, define
$$G_i:=\{v\in M_{\mathbb{R}}|<b_{i},v>+r_{i}\leq 0\}.$$
\begin{defn}
A circuit $S\subset \A$ is a minimal subset of hyperplanes satisfying $\cap_{i\in S}H_i=\emptyset$, and let $S:=S^+\sqcup S^-$ be the unique splitting such that 
$$(\cap_{i\in S^+}G_i)\cap(\cap_{j\in S^-}F_j)=\emptyset.$$
\end{defn}
Let $S=S^+\sqcup S^-$ be a minimal circuit, so that $\{b_{i_k}| i_k\in S\}$ linearly dependent in $\overline{N}$. 
Then there exists positive integers $w_i\in \zz_{>0}$ for $i\in S$ such that 
$$\sum_{i\in S^+}w_i\overline{b}_i-\sum_{j\in S^-}w_j\overline{b}_j=0.$$
Let 
$$\beta_S=\sum_{i\in S^+}w_ie_i-\sum_{j\in S^-}w_je_j,$$
then by exact sequence (\ref{exact1}) $\beta_S$ determines an element in $DG(\beta)^*=H_2(X_\A)$. 
\begin{example}
Let the stacky hyperplane arrangement $(N, \beta,\theta)$ be given by:
$$\mathbb{Z}^4\stackrel{\beta}{\longrightarrow}\zz^2,$$
and $\beta$ is by 
$$
\begin{cases}
b_1=(1,0);\\
b_2=(0,-1);\\
b_3=(0,1);\\
b_4=(-1,-1).
\end{cases}
$$
The generic element $\theta=(1,1)\in DG(\beta)=\zz^2$.  
The normal fan of the bounded polytope $\Gamma$ is the toric fan of a Hirzebruch surface. 
The core of the hypertoric Deligne-Mumford stack is a union of Hirzebruch surface and a projective plane. 
We have three circuits in this case:
$$S=(1,2,4);\quad S=(1,3,4);\quad S=(2,3).$$
\begin{enumerate}
\item $S=S^+\sqcup S^-=(1,4)\sqcup (2)$.  Then $\beta_S=e_1+e_4-e_2=(1,-1,0,1)$.  Then it determines an element 
$(0,1)\in \zz^2$;
\item $S=S^+=(1,3,4)$.  Then $\beta_S=e_1+e_3+e_4=(1,0,1,1)$.  Then it determines an element 
$(1,1)\in \zz^2$;
\item $S=S^+=(2,3)$.  Then $\beta_S=e_2+e_3=(0,1,1,0)$.  Then it determines an element 
$(1,0)\in \zz^2$.
\end{enumerate}
\end{example}

\begin{example}
Look at Example \ref{example1:core} again. 
We have three circuits in this case:
$$S=(1,2,4);\quad S=(1,3,4);\quad S=(2,3).$$
\begin{enumerate}
\item $S=S^+\sqcup S^-=(1,4)\sqcup (2)$.  Then $b_1+b_4-2b_2=0$ and $\beta_S=e_1+e_4-2e_2=(1,-2,0,1)$.  Then it determines an element 
$(0,1)\in \zz^2$;
\item $S=S^+=(1,3,4)$.  Then $b_1+2b_3+b_4=0$ and  $\beta_S=e_1+2e_3+e_4=(1,0,1,1)$.  Then it determines an element 
$(2,1)\in \zz^2$;
\item $S=S^+=(2,3)$.  Then $b_2+b_3=0$ and $\beta_S=e_2+e_3=(0,1,1,0)$.  Then it determines an element 
$(1,0)\in \zz^2$.
\end{enumerate}
\end{example}

\begin{thm}
The $\bbT$-equivariant cohomology ring of $X_\A$ is given by:
$$H^*_{\bbT}(X_\A)\cong \frac{R_\bbT [u_1, \cdots, u_m, \hbar]}{\{\langle \prod_{i\in S^+}u_i\cdot \prod_{j\in S^-}(\hbar-u_j)\rangle | S ~\mbox{a circuit}\}}.$$
\end{thm}
\begin{proof}
Recall that $X_\A$ is a closed substack of the Lawrence toric Deligne-Mumford stack 
$\mathcal{X}(\mathbf{\Sigma_{\theta}})$. From the main result in \cite{JT2}, the cohomology of $X_\A$ is isomorphic to the cohomology ring of  $\mathcal{X}(\mathbf{\Sigma_{\theta}})$.   There is a torus $\bbT$ action on  $\mathcal{X}(\mathbf{\Sigma_{\theta}})$ and the $\bbT$-equivariant cohomology ring of $X_\A$ is isomorphic to the $\bbT$-equivariant cohomology ring of  $\mathcal{X}(\mathbf{\Sigma_{\theta}})$. 
We have:
$$H_{\bbT}^*(\mathcal{X}(\mathbf{\Sigma_{\theta}}))\cong \frac{R_\bbT [u_1,\cdots,u_m,v_1,\cdots,v_m]}{I_\theta},$$
where 
$$I_\theta=\bigcap_{C}\langle \sigma(C,\theta)\rangle\subset \qq[u,v]$$
and
$C=\{a_{i_1}, \cdots, a_{i_d}\}$, $\lambda_1a_{i_1}+\cdots+\lambda_d a_{i_d}=\theta$ so that 
$$\sigma(C,\theta)=\{u_{i_j}: \lambda_j>0\}\cup\{u_{i_j}: \lambda_j<0\}.$$
Now we put into the extra $\cc^\times$-action on the fibre of $T^*\cc^m$.  The line bundle $L_i^{-1}$ over $\mathcal{X}(\mathbf{\Sigma_{\theta}})$
will admit such a $\cc^\times$ action. Look at the ideal 
$$I_\theta=\bigcap\langle C, \theta\rangle,$$
as in \cite[Section 4]{HS}, $I_\theta$ corresponds to all 
$$\{S=(i_1,\cdots,i_k)| \overline{b}_{i_1}, \cdots, \overline{b}_{i_k} ~\text{linearly dependent in}~ \overline{N}\}.$$
Now let $S=S^+\sqcup S^-$ be the decomposition such that 
$i_j\in S^+$ if $\lambda_j>0$ and $i_j\in S^-$ if $\lambda_j<0$.  Note that in equivariant setting we have 
$v_i=\hbar-u_i$. So 
the formula in the theorem just follows. 
\end{proof}

\subsubsection{$\bbT$-equivariant Chen-Ruan cohomology.}

The Chen-Ruan cohomology of $X_\A$ was calculated algebraically in \cite{JT} and symplectically in \cite{GH}. Here we calculate the {\em $\bbT$-equivariant} Chen-Ruan cohomology of $X_\A$.

Recall that the twisted sectors (components of the inertia stack $IX_\A$) of the hypertoric Deligne-Mumford stack $X_\A$ are given by 
the box elements $\mbox{Box}(\Delta_\beta)$ of the multi-fan $\Delta_\beta$. 
The set $\mbox{Box}(\Delta_{\mathbf{\beta}})$ consists of all 
pairs $(v,\sigma)$, where $\sigma$ is a cone in the multi-fan
$\Delta_{\mathbf{\beta}}$,  $v\in N$ such that 
$\overline{v}=\sum_{\rho_{i}\subset\sigma}\alpha_{i}\overline{b}_{i}$ for 
$0<\alpha_{i}<1$.  For $(v,\sigma)\in
\mbox{Box}(\Delta_{\mathbf{\beta}})$ we consider a closed substack of
$X_\A$ given by the quotient stacky hyperplane arrangement
$\mathcal{A}(\sigma)$. The inertia stack of  $X_\A$ is
the disjoint union of all such closed substacks, see Example \ref{inertia:stacks} or \cite[Section 4]{JT}.

We introduce a variable $\mathds{1}_{(v,\sigma)}$ for each box element $(v,\sigma)$.  We make a convention that 
$\mathds{1}_{(v,\sigma)}=1$ if $(v,\sigma)=(0,0)$ is the trivial box element. 
Let $(v,\sigma)\in \mbox{Box}(\Delta_{\mathbf{\beta}})$, say $v\in N(\tau)$ for
some top dimensional cone $\tau$. Let $(\check{v},\sigma)\in
\mbox{Box}(\Delta_{\mathbf{\beta}})$ be the inverse of $v$ as an element
in the group $N(\tau)$. Equivalently, if
$\overline{v}=\sum_{\rho_{i}\subseteq \sigma}\alpha_{i}\overline{b}_{i}$
for $0<\alpha_{i}<1$, then $\check{\overline{v}}=\sum_{\rho_{i}\subseteq
\sigma}(1-\alpha_{i})\overline{b}_{i}$. 

For any two box elements $(v_1, \tau_1), (v_2,\tau_2)\in \mbox{Box}(\Delta_\beta)$, 
if $\tau_1\cup\tau_2$ is a cone in $\Delta_\beta$, then there is a $(v_{3},\sigma_{3})$ which is  unique  in $\mbox{Box}(\Delta_{\mathbf{\beta}})$  such that $v_{1}+v_{2}+v_{3}\equiv 0$  in the local group given by $\sigma_{1}\cup\sigma_{2}$. 
Let $\overline{v}_{i}=\sum_{\rho_{j}\subseteq \sigma_i} \alpha^{i}_j \overline{b}_j$, with $0<\alpha^{i}_j<1$ and
$i=1,2,3$. The existence of  $\alpha^{1}_j,\alpha^{2}_j,\alpha^{3}_j$  means that $\rho_j$ belongs to $\sigma_1, \sigma_2, \sigma_3$.
Let $\sigma_{123}$ be the cone in $\Delta_{\mathbf{\beta}}$ such that 
$\overline{v}_{1}+\overline{v}_{2}+\overline{v}_{3}=\sum_{\rho_i\subseteq\sigma_{123}}a_{i}\overline{b}_i$, with $a_i=1$ or $2$. Let $I$ be the set of 
$i$ such that $a_i=1$ and  $\alpha^{1}_j, \alpha^{2}_j, \alpha^{3}_j$ exist, $J$ the set of $j$ such that $\rho_j$ belongs to $\sigma_{123}$ but not to $\sigma_3$.

\begin{thm}
 The $\bbT$-equivariant Chen-Ruan cohomology $H^*_{\CR,\bbT}(X_\A)$ is given by
$$H^*_{\CR,\bbT}(X_\A)=\frac{R_\bbT [u_1,\cdots,u_m, \hbar, \{\mathds{1}_{(v,\sigma)}\}_{(v,\sigma)\in \bbox(\Delta_\beta)}]}{I+J},$$
where
\begin{enumerate}
\item $I$ is the ideal of all products 
$\langle \prod_{i\in S^+}u_i\cdot\prod_{j\in S^-}(\hbar-u_j)\rangle$ for all circuits $S\subset \A$;
\item  $J$ is the ideal generated by the relations:
\begin{equation}\label{Chen-Ruan:relations}
\begin{cases}
\mathds{1}_{(v,\tau)}\cdot u_i=0,& \tau\cup\rho_i ~\text{is not a cone in}~ \Delta_\beta;\\
\mathds{1}_{(v_1,\tau_1)}\cdot \mathds{1}_{(v_2,\tau_2)}=(-1)^{|I|+|J|}\mathds{1}_{(v_3,\tau_3)}\cdot\prod_{i\in I}u_i\cdot\prod_{i\in J}u_i^2, & \tau\cup\rho_i ~\text{is a cone in}~ \Delta_\beta ~\text{and}~ \overline{v}_1\neq\check{\overline{v}}_2;\\
\mathds{1}_{(v_1,\tau_1)}\cdot \mathds{1}_{(v_2,\tau_2)}=(-1)^{|J|}\cdot\prod_{i\in J}u_i^2, & \tau\cup\rho_i ~\text{is a cone in}~ \Delta_\beta ~\text{and}~ \overline{v}_1=\check{\overline{v}}_2;\\
\mathds{1}_{(v_1,\tau_1)}\cdot \mathds{1}_{(v_2,\tau_2)}=0, & \tau\cup\rho_i ~\text{is not a cone in}~ \Delta_\beta.
\end{cases}
\end{equation}
\end{enumerate}
\end{thm}
\begin{proof}
Let $M_\beta$ be the matroid determined by $\beta$. Then $\qq[M_\beta]$ is the cohomology ring of hypertoric Deligne-Mumford stack 
$\M$ and we write:
$$H^*_{\bbT}(X_\A)\cong R_\bbT [M_\beta]=\frac{R_\bbT [u_1,\cdots,u_m,\hbar]}{I_{M_\beta}}.$$
By the definition of Chen-Ruan cohomology
$$H^*_{\CR, \bbT}(X_\A)=\bigoplus_{(v,\sigma)\in \bbox(\Delta_\beta)}H^*_{\bbT}(X_{\A/\sigma}),$$
where $H^*_{\bbT}(X_{\A/\sigma})$ is the $\bbT$-equivariant cohomology of the twisted sector 
$X_{\A/\sigma}$.  According to \cite{JT}, $H^*_{\bbT}(X_{\A/\sigma})$  is isomorphic to 
$R_\bbT [M_{\beta(\sigma)}]$, the equivariant cohomology ring of the induced matroid 
$M_{\beta(\sigma)}$.  By the proof of \cite[Proposition 5.8]{JT}, 
$R_\bbT [M_{\beta(\sigma)}]\cong \mathds{1}_{(v,\sigma)}\cdot R_\bbT [M_\beta]$. 
Then as a vector space, we have 
$$H^*_{\CR,\bbT}(X_\A)=\bigoplus_{(v,\sigma)\in \bbox(\Delta_\beta)}\mathds{1}_{(v,\sigma)}\cdot R_\bbT [M_\beta].$$

Then to prove the formula in the theorem, it is sufficient to prove the orbifold cup product of the box elements and the generators $u_i$ satisfy the relation in $(2)$, but all of these relations come from the definition of orbifold cup product calculated in \cite[Lemma 5.13]{JT} and in the proof of \cite[Theorem 1.1]{JT} in \cite[Section 5.3]{JT}.
\end{proof}

\subsection{Gromov-Witten theory}\label{GW:theory}

We briefly review some basic definitions on orbifold Gromov-Witten theory and the reduced virtual fundamental cycles.

\subsubsection{Equivariant Gromov-Witten invariants}
Let $X$ be a smooth Deligne-Mumford stack, endowed with a torus $T$ action.  The moduli stack 
$\cM_{0,n}(X, d)$ of degree $d\in H_2(X, \qq)$ twisted stable maps to $X$ carries a $T$-action, and a 
virtual fundamental cycle $[\cM_{0,n}(X,d)]^{\virt}\in A_{*, T}(\cM_{0,n}(X,d))$.
There are $T$-equivariant evaluation maps\footnote{We ignore the issue of trivializing the marked gerbes in our moduli problem. A detailed discussion on this can be found in \cite{AGV}.}:
$$\ev_i: \cM_{0,n}(X,d)\to IX$$
to the inertia stack $IX$ of $X$ for $1\leq i\leq n$, see \cite{CR2}, \cite{AGV}. 

Given $\gamma_1,...,\gamma_n\in H_{\CR, T}^*(X)$, we consider the following genus $0$ $T$-equivariant Gromov-Witten invariant:
$$\langle \gamma_1,\cdots,\gamma_n\rangle_{0,n,d}^{X}=\int^T_{[\cM_{0,n}(X,d)]^{\virt}}\prod_{i}\ev_i^\star\gamma_i.$$

The moduli stack $\cM_{0,n}(X,d)$ has components indexed by the components of inertia stack $IX$. 
We write 
$$IX=\bigsqcup_{f\in \sf B}X_f$$
for the decomposition of $IX$ into connected components, where $\sf B$ is the index set. 
Then the component $\cM_{0,n}(X,d)^{f_1,\cdots,f_n}$ is the one which under evaluation maps $\ev_i$, the images lie in
the component $X_{f_i}$. 
The virtual dimension of $\cM_{0,n}(X,d)^{f_1,\cdots,f_n}$ is:
\begin{equation}\label{virtual:dimension}
-K_{X}\cdot d+\dim(X)+n-3-\sum_{i}\age(X_{f_i}).
\end{equation}

If $X$ is not compact (like our hypertoric Deligne-Mumford stacks), then the moduli stack $\cM_{0,n}(X, d)$ is non-compact. 
There is a $T$-action on $\cM_{0,n}(X, d)$.  Assume that the $T$-fixed locus $\cM_{0,n}(X, d)^T$ is compact, then $T$-equivariant GW invariants can be defined in the same way, replacing equivariant integration by equivariant residues.


Let $\NE(X)\subset H_2(X, \rr)$ be the cone generated by classes of effective curves and set 
$$\NE(X)_{\zz}:=\{d\in H_2(X,\zz): d\in \NE(X)\}.$$
Let $R_T:=H^*_{T}(\pt)$ and $R_T[\![Q]\!]$ the formal power series ring
\[
R_T[\![Q]\!] = \left\{ \sum_{d \in \tiny\NE(X)_\zz} a_d Q^d : a_d \in R \right\}
\]
so that $Q$ is a so-called \emph{Novikov variable}~\cite[III~5.2.1]{Manin}.  For $\gamma_i, \gamma_j, t\in H_{\CR,T}^*(X)$, the big $\bbT$-equivariant quantum product is defined by:
\begin{equation}\label{big:quantum:product}
(\gamma_i\star_{\bbig,t} \gamma_j, \gamma_k)=\sum_{d\in \tiny\NE(X)_{\zz}}\sum_{n\geq 0}Q^d\langle \gamma_i, \gamma_j, \underbrace{t,\cdots,t}_{n}, \gamma_k\rangle_{0,n+3,d}^{X}
\end{equation}
The small $\bbT$-equivariant quantum product is defined by putting $n=0$:
\begin{equation}\label{quantum:product2}
(\gamma_i\star \gamma_j, \gamma_k)=\sum_{d\in \tiny\NE(X)_{\zz}}Q^d\langle \gamma_i, \gamma_j, \gamma_k\rangle_{0,3,d}^{X}
\end{equation}
or 
\[
\gamma_i\star \gamma_j=\sum_{d\in \tiny\NE(X)_{\zz}}Q^d\cdot \inv^\star\cdot \ev_{3,\star}(\ev_1^\star(\gamma_i)\ev_j^\star(\gamma_j)\cap [\cM_{0,3}(X,d)]^{\virt})
\]
where $\inv: IX\to IX$ denotes the involution sending $(x, g)\mapsto (x, g^{-1})$, for 
$x\in X, g\in \Aut(x)$. 
The quantum product satisfies the associativity property and makes $H_{\CR,T}^*(X)\otimes R_T[\![Q]\!]$ a ring, which is called the small equivariant  quantum cohomology ring. 

Assume that $d\neq 0$ and $D\in H^2(X,\qq)$ a divisor class. Then the {\em divisor equation} of Gromov-Witten invariants is:
\begin{equation}\label{Divisor:Equation}
\langle D, \gamma_i, \gamma_j\rangle_{0,3,d}^{X}=(D\cdot d)\cdot \langle\gamma_i, \gamma_j\rangle_{0,2,d}^{X}.
\end{equation}
So in order to calculate the quantum product by a divisor $D$, it suffices to study two point Gromov-Witten invariants of $X$.

\subsubsection{Reduced virtual fundamental cycle}\label{reduced:virtual:cycle}

Let $X$ be a smooth quasi-projective Deligne-Mumford stack with a nowhere vanishing holomorphic symplectic form $\omega\in \Omega^{0,2}(X)$.  By \cite{KL}, the usual non-equivariant virtual fundamental cycle vanishes when $d\neq 0$. The hypertoric Deligne-Mumford stack $X_\A$ admits such a holomorphic symplectic form $\omega$ induced from the canonical symplectic form on $T^*\cc^m$. In this situation we need to consider reduced virtual fundamental classes in order to obtain non-trivial invariants.

We recall the reduced virtual fundamental cycle construction following \cite[Section 4.2]{BMO} (this is a special case of cosection localized cycle \cite{KL}). 
Let $\cC$ be a fixed twisted nodal curve, and let $\cM_{\cC}:=\cM_{\cC}(X,d)$ denote the moduli stack of twisted maps from 
$\cC$ to $X$ with degree $d$.  Recall that in \cite{Behrend}, the obstruction theory for $\cM_{\cC}$ is given by:
$$R\pi_{\star}(\ev^\star T_{X})^{\vee}\to L_{\cM_{\cC}},$$
where $L_{\cM_{\cC}}$ is the cotangent complex of $\cM_{\cC}$, and 
$$\ev: \cC\times \cM_{\cC}\to IX$$
$$\pi: \cC\times \cM_{\cC}\to \cM_{\cC}$$
are evaluation maps and projection to $\cM_{\cC}$, respectively. 

Denote by $\omega_{\pi}$ the relative dualizing sheaf.  Pairing with the symplectic form and pullback differentials we have a map:
$$\ev^\star(T_{X})\to \omega_{\pi}\otimes (\cc\omega)^*.$$
Then this induces a morphism
$$R\pi_{\star}(\omega_{\pi})^\vee\otimes \cc\omega\to R\pi_{\star}(\ev^\star(T_{X})^{\vee}).$$
Taking truncation we get a morphism:
$$\iota: \tau_{\leq -1}R\pi_{\star}(\omega_{\pi})^\vee\otimes \cc\omega\to R\pi_{\star}(\ev^\star(T_{X})^{\vee}).$$
The truncation is a trivial line bundle, but carries a nontrivial action in the equivariant setting depending on the $T$-action on $X$.

There is an induced map from the mapping cone of $\iota$: 
\begin{equation}\label{reduced:obstruction:theory}
C(\iota)\to L_{\cM_{\cC}}
\end{equation}
which satisfies the conditions in the perfect obstruction theory of \cite{BF}, \cite{LT}.  The reduced virtual fundamental cycle of 
$\cM_\cC$ is defined by the reduced obstruction theory (\ref{reduced:obstruction:theory}). The reduced virtual fundamental class $[\cM_{0,n}(X, d)]^{\red}$ for $\cM_{0,n}(X,d)$ is obtained by applying the construction above to the universal family over $\cM_{0,n}(X,d)$.

In the hypertoric Deligne-Mumford stack $X_\A$ case, there is a $\bbT$-action on $X_\A$ and the moduli stack 
$\cM_{0,n}(X_\A,d)$.   The extra $\cc^\times$ acts on the space of symplectic forms $\cc\omega$ with nontrivial weights. Hence from
(\ref{reduced:obstruction:theory}), 
\begin{equation}\label{reduced:virtual:class:virtual:class}
[\cM_{0,n}(X_\A,d)]^{\virt}=\hbar [\cM_{0,n}(X_\A,d)]^{\red}.
\end{equation}
The detail argument of this relation can be found in \cite[Section 4.2]{BMO}. Although \cite{BMO} works with smooth schemes, the arguments hold true for smooth Deligne-Mumford stacks. 

\subsubsection{Maps from $\pp^{1}_{s_1,s_2}$ to $X_{\A}$}

Let $\pp^{1}_{s_1,s_2}$ be the unique  $\pp^1$-orbifold with stacky points $P_1=[1,0]=B\mu_{s_1}$ and $P_2=[0,1]=B\mu_{s_2}$, and no other stacky points. For the purpose of calculation, we classify the morphisms from $\pp^{1}_{s_1,s_2}$ to $X_{\A}$. 

Recall that our hypertoric Deligne-Mumford stack $X_\A$ is open, and the core is a union of toric Deligne-Mumford stacks given by the bounded polytope of the stacky hyperplane arrangement, see \S \ref{core:hypertoricDMstack}.  A morphism $\pp^{1}_{s_1,s_2}\to X_{\A}$ must have image in an irreducible component of the core, i.e. must lie in a toric Deligne-Mumford stack inside the core. 
So an argument similar to \cite[Section 3]{CCIT} works for hypertoric Deligne-mumford stacks.
 
The torus $\bbT$-fixed points of $X_\A$ are all isolated, and have an one-to-one correspondence with the top dimensional cones in $\Delta_{\beta}$.  
For $\sigma\in \Delta_{\beta}$ a top dimensional cone, let $(X_\A)_{\sigma}$ denote the fixed point corresponding to $\sigma$. 
Let $\sigma, \sigma^\prime\in \Delta_{\beta}$ be two top dimensional cones,  we write $\sigma|\sigma^\prime$ if they intersect along a codimension-1 face.  Denote by $j$ the unique index such that $\overline{b}_j\in\sigma$,
$\overline{b}_j\notin\sigma^\prime$; and $j^\prime$  such that $\overline{b}_{j^\prime}\in\sigma^\prime$,
$\overline{b}_{j^\prime}\notin\sigma$. 

Recall the box element $(f,\tau)\in\mbox{Box}$ is given by 
$f=\sum_{i=1}^{m}f_i\overline{b}_i$ for $0\leq f_i\leq 1$. 
Note that $f_i=0$ if $i\notin\tau$. 

The following is analogous to \cite[Proposition 10]{CCIT}:
\begin{prop}\label{classification:P1orbifold}
Let $X_\A$ be the hypertoric Deligne-Mumford stack associated to a stacky hyperplane arrangement $\A$. Suppose that 
$\sigma, \sigma^\prime\in \Delta_{\beta}$ such that $\sigma|\sigma^\prime$, and 
$(f,\tau)\in\bbbox$ such that $\tau\subset \sigma$. Then the following are equivalent:
\begin{enumerate}
\item  A representable morphism $g: \pp^1_{s_1,s_2}\to X_\A$ such that 
$f(0)=(X_\A)_{\sigma}$, $f(\infty)=(X_\A)_{\sigma^\prime}$, and the restriction
$g|_{0}: B\mu_{s_1}\to (X_\A)_{\sigma}$ gives $f$.
\item A positive rational number $c$ such that $\langle c\rangle=f_j$.
\end{enumerate}
\end{prop}
\begin{proof}
Since any morphism $\pp^1_{s_1,s_2}\to X_\A$ must have image inside an irreducible component of the core, which is a toric Deligne-Mumford stack, we are reduced to the case in \cite[Proposition 10]{CCIT}. 
\end{proof}
\begin{rmk}
As in \cite[Remark 11]{CCIT}, the morphism in the above Proposition actually determines both $s_2$ and the box element 
$f^\prime$ given by the restriction $g|_{\infty}: B\mu_{s_2}\to (X_\A)_{\sigma^\prime}$. Moreover they satisfy the following 
relation 
\begin{equation}\label{degree:P1:to:XA}
f+\lfloor c\rfloor b_j +q^\prime b_{j^\prime}+f^\prime \equiv 0 \text{~mod~} \bigoplus_{i\in\sigma\cap\sigma^\prime}\zz b_i
\end{equation}
for some $q^\prime\in \zz_{\geq 0}$.

Hypertoric Deligne-Mumford stack $X_\A$ can be deformed into a hypertoric Deligne-Mumford stack $X_\lambda$ corresponding to subregular parameter $\lambda$ (see \S \ref{deformation:hypertoricDM:stack}), so that the curve classes are all lie in closed substacks $\cM^S$ associated to circuits $S$.  So any morphism $\pp^1_{s_1,s_2}\to X_\A$ must have image inside $\cM^S$, hence actually inside the weighted projective stack $\pp^{|S|-1}_{\mathbf{w}}$.  The above degree formula (\ref{degree:P1:to:XA}) can be written down in terms of the weights $\mathbf{w}=(w_1, \cdots,w_n)$.
\end{rmk}

\section{Deformation of hypertoric Deligne-Mumford stacks}\label{Section:deformation:hypertoricDM}

In this section we consider the deformation of symplectic Deligne-Mumford stacks in the hypertoric case in order to compute its Gromov-Witten invariants. For this purpose, we first recall the construction of hypertoric Deligne-Mumford stacks in symplectic category by \cite{HP}, \cite{GH}. 

\subsection{Symplectic construction}
 
In \S \ref{Section:hypertoricDMstack}, the hypertoric Deligne-Mumford stack is constructed as a quotient stack $[Y/G]$, where $Y\subset T^*\cc^m$ is a locally closed subvariety.  Hypertoric Deligne-Mumford stack in symplectic category can be constructed in a similar way. 
Let $\hh:=T^*\cc^m$ be the cotangent bundle of $\cc^m$, which is a hyperk\"ahler manifold with a canonical holomorphic symplectic form $\omega_{\cc}$.  The torus $T=(\cc^\times)^m$ acts on $\hh$ by diagonal on $\cc^m$ and minus the diagonal on the fibre. 

We introduce the hyperk\"ahler $T$-moment map $\widetilde{\mu}=(\widetilde{\mu}_{\rr}, \widetilde{\mu}_{\cc})$ on $\hh$.  Let 
$\{t_i\}_{i=1}^{m}$ be a dual basis to the basis $\{\epsilon_i\}_{i=1}^{m}$ in $(\mathfrak{t}^m)^*$, where $\mathfrak{t}^m$ is the Lie algebra of $T$.  Then 
\begin{align}\label{eq:moment:map}
\widetilde{\mu}_{\rr}(z,w)=\frac{1}{2}\sum_{i=1}^{m}(\|z_i\|^2-\|w_i\|^2)t_i\in (\mathfrak{t}^m)^*;\\ \nonumber
\widetilde{\mu}_{\cc}(z,w)=\sum_{i=1}^{m}z_iw_i t_i\in (\mathfrak{t}_\cc^m)^*.
\end{align}
Recall in our stacky hyperplane arrangement $\A=(N,\beta,\theta)$, 
$\mathcal{H}=\{H_1,\cdots,H_m\}$ in $M_\rr$ defines a hyperplane arrangement (weighted hyperplane arrangement in \cite{GH})
in $(\mathfrak{t}^m)^*=M_\rr$.  Then there are exact sequences similar to the ones on \S \ref{Section:stacky:hyperplane:arrangement}:
\begin{equation}
0\rightarrow \mathfrak{t}^{m-d}\longrightarrow \mathfrak{t}^m\stackrel{\beta}{\longrightarrow} \mathfrak{t}^d\rightarrow 0,
\end{equation}
where $\beta(\epsilon_i)=\overline{b}_i$; and 
\begin{equation}
0\rightarrow (\mathfrak{t}^{d})^*\stackrel{\beta^\vee}{\longrightarrow} (\mathfrak{t}^m)^*\stackrel{\iota^*}{\longrightarrow} (\mathfrak{t}^{m-d})^*\rightarrow 0,
\end{equation}
where $\iota^*(t_i)=\lambda_i$.  The $T$-moment map in (\ref{eq:moment:map}) induces a subtorus $(\cc^\times)^{m-d}$-moment map by:
\begin{align}\label{eq:moment:map:subtorus}
\mu_{\rr}(z,w)=\iota^*\left(\frac{1}{2}\sum_{i=1}^{m}(\|z_i\|^2-\|w_i\|^2)t_i\in (\mathfrak{t}^m)^*\right);\\ \nonumber
\mu_{\cc}(z,w)=\iota^*\left(\sum_{i=1}^{m}z_iw_i t_i\in (\mathfrak{t}_\cc^m)^*\right).
\end{align}
The hyperk\"ahler moment map is surjective onto $(\mathfrak{t}^m)^*\oplus(\mathfrak{t}_\cc^m)^*$. 
Let $(\theta,\lambda)\in (\mathfrak{t}^{m-d})^*\oplus(\mathfrak{t}_\cc^{m-d})^*$ be a regular value.  The hypertoric Deligne-Mumford stack
(actually in this case hypertoric orbifold) is then defined by
$$\mathcal{M}(\mathcal{H})=\hh\mmod_{(\theta,\lambda)}T^{m-d}$$
which is the hyperk\"ahler reduction of $\hh$ by $T^{m-d}$.  We think of this as a GIT quotient stack
$$\mathcal{M}_{\lambda}(\mathcal{H})=[\mu_{\cc}^{-1}(\lambda)/_{\theta}T^{m-d}]$$
by the stability parameter $\theta$. 
\begin{rmk}
The algebraic version of hypertoric Deligne-Mumford stack $X_\A$ defined in \S \ref{Section:hypertoricDMstack} is the case that $\lambda=0$. 

In general, if $N$ has torsion $N_{\mbox{\tiny tor}}=\mu$, which is a finite abelian group, $X_\A$ is a $\mu$-gerbe over the orbifold 
$\MHL$. 
\end{rmk}

\subsection{Deformations}\label{deformation:hypertoricDM:stack}

We adopt the argument by Konno \cite{Konno} as reviewed in \cite[Section 4.1]{MS}. 

According to \cite{Konno} and \cite[Section 4.1]{MS}, we call a parameter $\lambda\in (\mathfrak{t}_\cc^{m-d})^*$ {\em sub-regular} if $\lambda$ lies on a unique root hyperplane
$$K_S=\mbox{span}(\iota^*e_i^\vee: i\notin S)$$
for $S\subset\A=\{1,\cdots,m\}$.  It is easy to check that $S$ is a circuit. 
\begin{prop}[\cite{Konno}, Theorem 5.10]\label{Konno:stable}
Let $(z,w)\in \mu_\cc^{-1}(\lambda)$. Then $(z,w)$ is $\theta$-stable if and only if either if the following conditions hold:
\begin{enumerate}
\item $z_i\neq 0$ for some $i\in S^+$;
\item $w_i\neq 0$ for some $i\in S^-$.
\end{enumerate}
\end{prop}

For a sub-regular $\lambda$, the hypertoric Deligne-Mumford stack $X_\A$ can be defined to be 
$\MHL$ with good properties:
\begin{prop}\label{deformation:hypertoric:subregular}
$\MHL$ contains a codimension $|S|-1$ substack $\cM^{S}$, which is a weighted projective 
$\pp^{|S|-1}_{\mathbf{w}}$-bundle over an affine hypertoric stack $\cM_0^{S}$ for a circuit $S\subset \A$. All positive dimensional projective substacks in $\MHL$ are contained in $\cM^{S}$. 
\end{prop}
\begin{proof}
Let $S\subset \A$ be a circuit. Define the following space:
$$\mathcal{P}^S:=\{w_i=0: i\in S^+; z_i=0: i\in S^-\}\subset \hh.$$
The substack $\mathcal{M}^S$ is the GIT quotient stack
$$\cM^S=[(\mathcal{P}^S\cap \mu_\cc^{-1}(\lambda))/_{\theta}T^{m-d}].$$
Now let 
$$p: \mathfrak{t}_\cc^m\to \cc^{m-|S|}$$
denote the projection to the last $m-|S|$ coordinates. Then the following are true:
Recall that $$\beta_S=\sum_{i\in S^+}w_i e_i-\sum_{j\in S^-}w_j e_j.$$
Then we have:
\begin{enumerate}
\item $\ker_{\mathfrak{t}_\cc^{m-d}}=\cc\cdot \beta_S$;
\item  $(p(\mathfrak{t}_\cc^{m-d}))^*$ is canonically identified with $K_S\subset (\mathfrak{t}_\cc^{m-d})^*$;
\item $\lambda\in (p(\mathfrak{t}_\cc^{m-d}))^*$.
\end{enumerate}
Consider 
$$p|_{T=(\cc^\times)^m}: T\to (\cc^\times)^{m-|S|}$$
and 
$$ \hh\to T^*\cc^{m-|S|}$$
by $(z_i, w_i)\mapsto (z_i, w_i)_{i\notin S}$. 
We have 
$p|_{T}(T^{m-d})$ acts on $T\to (\cc^\times)^{m-|S|}$ with moment map $\mu_\cc|_{m-|S|}$. 
So the hypertoric Deligne-Mumford stack is 
$$\cM_0^S=[\mu_\cc|_{m-|S|}^{-1}(\lambda)/_{\theta}T^{m-d}].$$
Since $\lambda\in K_S$, $\theta$ is zero and the stack is affine.

By Theorem \ref{Konno:stable}, if $(z,w)\in (\mathcal{P}^S\cap \mu_\cc^{-1}(\lambda))$, then 
$p(z,w)\in \mu_\cc|_{m-|S|}^{-1}(\lambda)$.  Then we have a morphism of quotient stacks
$$\eta_S: \cM^S\to \cM_0^{S}.$$
Let $\cc^{|S|}=\{z_i: i\in S^+; w_i: i\in S^-\}$. 
The fibre of $\eta_S$ is isomorphic to the GIT quotient stack 
$[\cc^{|S|}/_{\theta}\cc^\times]$, where $\cc^\times=\ker: T^{m-d}\to p(T^{m-d})$.  Note that the action of $\cc^\times$ on $\cc^{|S|}$ is given by the weights $w_i$.  So it is the weighted projective stack $\pp^{|S|-1}_{\mathbf{w}}$. 
Any positive projective substacks of $\MHL$ corresponds to $T^{m-d}$ orbits in $\hh$ whose closure interests with the unstable locus, and they are contained in $\mathcal{P}^S$. Hence the positive dimensional substacks must be contained in $\cM^S$. 
\end{proof}

\section{Steinberg correspondence}\label{Section:Steinberg:correspondence}

In this section we discuss the stacky version of the  Steinberg correspondence.   

\subsection{General set-up}
From our construction of hypertoric Deligne-Mumford stack in \S \ref{Section:hypertoricDMstack}, one can take the hypertoric Deligne-Mumford stack $X_\A$ as a symplectic resolution of symplectic singularities. We give such a construction. 

Recall that $X_\A$ is a quotient stack 
$[Y/G]$, where $Y\subset \cc^{2m}\setminus V(\mathcal{I}_{\theta})$ is a closed subvariety of $\hh:=T^*\cc^m$, determined by  the ideal (\ref{ideal1}).  Let $\overline{Y}\subset \hh$ be the subvariety determined by the ideal $I_{\beta^{\vee}}$ in (\ref{ideal1}).  From the GIT point of view, $V(\mathcal{I}_{\theta})$ is the unstable locus so that deleting them we have a good GIT quotient. Consider the quotient stack 
$$X_\A^0=[\overline{Y}/G].$$
It is in general  a singular affine stack, and not Deligne-Mumford.  Let $\overline{X}_\A^0$ be the good moduli space of the stack 
$X_\A^0$ in the sense of \cite{Alper}.  
The inclusion $Y\hookrightarrow \overline{Y}$ induces a projective morphism of quotient stacks, which is birational. 
Then we take the natural morphism
$$X_\A\to \overline{X}_\A^0$$
as a ``symplectic resolution".  To simplify notation we denote by $X_0:=X_\A^0, \overline{X}_0=\overline{X}_\A^0$. 
The Steinberg variety is constructed by the fibre product
\begin{equation}\label{Steinberg:correspondence}
X_\A\times_{\overline{X}_0}X_\A.
\end{equation}
Each component $Z$ in (\ref{Steinberg:correspondence}) is of dimension $\leq \dim(X_\A)$ (see \cite{BMO}, \cite{Ginzburg}).  The components of dimension $\dim(X_\A)$ are Lagrangian inside $X_\A\times X_\A$.  
Let $\mathbf{Z}\subset X_\A\times X_\A$ be all the components of Lagrangian cycles. 
The Steinberg correspondence 
\begin{equation}
L:  H^*_{T\times\cc^\times}(X_\A)\to H^*_{T\times\cc^\times}(X_\A)
\end{equation}
for the $\bbT$-equivariant cohomology of $X_\A$ is given by:
$$L(\alpha)=p_{2\star}(\mathbf{Z}\cup p_1^\star\alpha),$$
where $p_i: X_\A\times X_\A\to X_\A$ is the projection for $i=1,2$.

In the smooth variety case, the Lagrangian components acts by correspondence on the Borel-Moore homology of $X_\A$. We hope that this  is still true for Chen-Ruan cohomology.  In this section we prove the case of cotangent bundle of weighted projective stacks, which we use for the calculation of quantum product.
Let 
$$[I\mathbf{Z}]\in H_*(I(X_\A\times X_\A))$$ 
be a cycle. Then  by Poincare duality it defines a cohomology class in 
$H_{\bbT}^*(I(X_\A\times X_\A))$. The Steinberg correspondence in Chen-Ruan cohomology is:
\begin{equation}\label{Steinberg:correspondence:Chen-Ruan}
\mathbf{L}:  H^*_{\CR, T\times\cc^\times}(X_\A)\to H^*_{\CR, T\times\cc^\times}(X_\A)
\end{equation}
for equivariant cohomology of $X_\A$ is given by:
$$\mathbf{L}(\alpha)=\inv^\star Ip_{2\star}([I\mathbf{Z}]\cup Ip_1^\star \alpha),$$
where $Ip_i: I(X_\A\times X_\A)\to IX_\A$ is the projection of inertia stacks  for $i=1,2$.

\subsection{The case of cotangent bundle $T^*{\pp^n_{\mathbf{w}}}$}\label{sec:cotangentbundle:Steinberg}

Let $X_\A=T^*\pp^n_{\mathbf{w}}$ be the cotangent bundle of weighted projective stack.
Then the Lagrangian components of (\ref{Steinberg:correspondence}) consists of two components:
$$X_\A;  \quad \pp^n_{\mathbf{w}}\times \pp^n_{\mathbf{w}}.$$
It is easy to see that they are Lagrangian in $X_\A\times X_\A$.  Since we are dealing with Chen-Ruan orbifold cohomology of $X_\A$,  we actually need to find a correspondence in orbifold version. 
For this purpose we consider the inertia stack $IX_\A$ and $I(\pp^n_{\mathbf{w}}\times \pp^n_{\mathbf{w}})$ inside the inertia stack
$I(X_\A\times X_\A)$.

Recall the the twisted sectors of weighted projective stack $\pp^n_{\mathbf{w}}$ corresponds to finite set 
$F=\{\frac{d}{w_i}| 0\leq d<w_i\}$.  The twisted sectors of the product $\mathbf{Z}:=\pp^n_{\mathbf{w}}\times \pp^n_{\mathbf{w}}$ are also indexed by 
pairs $(\alpha,\beta)$, where $\alpha=e^{\frac{d}{w_i}}$ and $\beta=e^{\frac{d^\prime}{w_j}}$. 
The pair $(\alpha,\alpha^{-1})$ means that 
they are inverse in local group, i.e. if $\alpha=e^{\frac{d}{w_i}}$, then $\alpha^{-1}=e^{\frac{w_i-d}{w_i}}$.
Set 
\begin{equation}
\Gamma=\bigcup_{(\alpha,\alpha^\prime)}\mathbf{Z}_{(\alpha,\alpha^{\prime})}.
\end{equation}
Let $[\mathbf{w}:d]:=\{w_i|d \text{~divides~}w_i\}$. Then $\pp([\mathbf{w}:d])$ is a sub-weighted projective stack of $\pp^n_{\mathbf{w}}$. 
and $\mathbf{Z}_{(\alpha,\alpha^{\prime})}=\pp([\mathbf{w}:d])\times\pp([\mathbf{w}: d^\prime])$. 

\begin{prop}
All of the twisted sectors of  $X_\A$ in $IX_\A$ have orbifold degree $\dim(X_\A)$ and $IX_\A$ gives an endomorphism   for the Chen-Ruan orbifold cohomology of $X_\A$. 

All of the components of $I(\pp^n_{\mathbf{w}}\times \pp^n_{\mathbf{w}})$ has orbifold degree $\dim(X_\A)$.  Furthermore, the cycle $\Gamma$ in the inertia stack defines an endomorphism  for $\bbT$-equivariant Chen-Ruan orbifold cohomology of $X_\A$.
\end{prop}
\begin{proof}
We first do the case $IX_\A$.
Note that all the components of $IX_\A$ are cotangent bundle of $T^*\pp([\mathbf{w}:d])$, where $d$ is a positive integer and $v=\frac{d}{w_i}$ is the local group element determining the twisted sector. The age of this element 
$$\age(v)=\overline{A}_d,$$
where $A_d$ denote the number of $\{w_i\}$'s such that $d|w_i$, and $\overline{A}_d$ is $n+1-A_d$. 
Note that $T^*\pp([\mathbf{w}:d])\subset (X_\A\times X_\A)_{(v,v)}=T^*\pp([\mathbf{w}:d])\times T^*\pp([\mathbf{w}:d])$ as a diagonal. 
Then the orbifold degree of $(X_\A)_{v}$ inside $I(X_\A\times X_\A)$ is:
$$2(A_d-1)+\overline{A}_d+\overline{A}_d=n+1+n+1-2=2n=\dim(X_\A).$$
Note that such a component in the correspondence:
$$
\xymatrix{
& X_\A\subset X_\A\times X_\A \ar[dr]^{p_1} \ar[dl]_{p_2} & \\ 
X_\A \ar@{}[rr]^{} & & X_\A
}
$$ 
always gives identity. So $X_\A$ gives an endomorphism  for the $\bbT$-equivariant Chen-Ruan cohomology. 

For the other component $\mathbf{Z}$, 
\begin{equation} 
\label{eq:Steinberg:correspondence} 
\xymatrix{
& \mathbf{Z}\subset X_\A\times X_\A \ar[dr]^{p_1} \ar[dl]_{p_2} & \\ 
X_\A \ar@{}[rr]^{} & & X_\A.
}
\end{equation} 
We prove that each component in $I\mathbf{Z}$ has orbifold degree of dimension of $X_\A$. 
A twisted sector of $\mathbf{Z}$ corresponds to a pair $(v_1, v_2)$ where 
$v_1=\frac{d_1}{w_i}$, $v_2=\frac{d_2}{w_j}$. 
And 
$$\mathbf{Z}_{(v_1,v_2)}=\pp([\mathbf{w}: d_1])\times\pp([\mathbf{w}: d_2])\subset (X_\A)_{v_1}\times (X_{\A})_{v_2}
=T^*\pp([\mathbf{w}: d_1])\times T^*\pp([\mathbf{w}: d_2]).$$
So the orbifold degree of $\mathbf{Z}_{(v_1,v_2)}$ is:
\begin{align*}
A_{d_1}-1+A_{d_2}-1+\age(v_1)+\age(v_2)&=A_{d_1}-1+A_{d_2}-1+\overline{A}_{d_1}+\overline{A}_{d_2}\\
&=n+1-1+n+1-1\\
&=2n.
\end{align*}
We may call all the components of $I\mathbf{Z}$ Lagrangian cycle in the correspondence of Chen-Ruan cohomology. 
Now each component in $\Gamma$ corresponds to 
\begin{equation} 
\label{eq:Steinberg:Gamma} 
\xymatrix{
& \mathbf{Z}_{(\alpha,\alpha^{\prime})}\subset (X_\A)_{\alpha}\times (X_\A)_{\alpha^{\prime}} \ar[dr]^{p_1} \ar[dl]_{p_2} & \\ 
(X_\A)_{\alpha} \ar@{}[rr]^{} & & (X_\A)_{\alpha^{\prime}}.
}
\end{equation} 
The above diagram actually corresponds to lower dimensional cotangent bundle of weighted projective stacks, hence similar argument as in \cite{BMO}, \cite{Ginzburg} proves that $\mathbf{Z}_{(\alpha,\alpha^{\prime})}$ gives an endomorphism  on 
$\bbT$-equivariant cohomology.  Chen-Ruan cohomology is a sum of all of the twisted sectors like this, hence $\Gamma$ gives an endomorphism  on $\bbT$-equivariant Chen-Ruan cohomology. 
\end{proof}

\begin{example}
We compute an example $X_\A=T^*\pp(1,2)$. 
In this example the components of $I\mathbf{Z}=I(\pp(1,2)\times\pp(1,2))$ are given by:
$$
\begin{array}{l}
\mathbf{Z}_{(0,0)}=\mathbf{Z}; \quad
\mathbf{Z}_{(0,1/2)}=\pp(1,2)\times B\mu_2;\\
\mathbf{Z}_{(1/2,0)}=B\mu_2\times \pp(1,2);\quad
\mathbf{Z}_{(1/2,1/2)}=B\mu_2\times B\mu_2.
\end{array}
$$
It is not hard that we can check that $\Gamma$ gives an endomorphism  for the Chen-Ruan cohomology of $X_\A$. 

We compute that using $I\mathbf{Z}$. Consider
$$
\xymatrix{
& \mathbf{Z}_{(\alpha,\beta)}\subset (X_\A)_{\alpha}\times (X_\A)_{\beta} \ar[dr]^{p_1} \ar[dl]_{p_2} & \\ 
(X_\A)_{\alpha} \ar@{}[rr]^{} & & (X_\A)_{\beta}.
}
$$
We apply the above diagram to every component in $I\mathbf{Z}$. The equivariant Chen-Ruan cohomology of $X_\A$ is:
$$H^*_{\CR,\bbT}(X_\A)=\frac{\qq[u_1, u_2, \hbar, \mathds{1}_{\frac{1}{2}}]}{\{u_1\cdot u_2, \mathds{1}_{\frac{1}{2}}u_1, \mathds{1}_{\frac{1}{2}} u_2, \mathds{1}_{\frac{1}{2}}\cdot \mathds{1}_{\frac{1}{2}}=u_1^2\}}.$$
We calculate:
\begin{align*}
\mathbf{L}(u_1)=\frac{1}{2}(\hbar-u_1-u_2)+\frac{1}{2}\mathds{1}_{\frac{1}{2}};\\
\mathbf{L}(u_2)=(\hbar-u_1-u_2)+\frac{1}{2}\mathds{1}_{\frac{1}{2}};\\
\mathbf{L}(\mathds{1}_{\frac{1}{2}})=\frac{1}{2}(\hbar-u_1-u_2)+\frac{1}{2}\mathds{1}_{\frac{1}{2}}.
\end{align*}
We calculate $\mathbf{L}^{-1}$, where 
$$\mathbf{L}^{-1}:=(Ip_1)_\star([I\mathbf{Z}]\cup Ip_2^\star(-)).$$
To do that we use localization. There are two $T\times\cc^\times$ fixed points on $X_\A$, which is 
$P_1=[1,0]\in \pp(1,2), P_2=[0,1]\in \pp(1,2)$.  We calculate:
$$e(N_{P_1}X_\A)=\lambda_1(\hbar-\lambda_1-\lambda_2);\quad e(N_{P_2}X_\A)=\lambda_2(\hbar-\lambda_1-\lambda_2).$$
We first calculate:
\begin{align*}
\int_{X_\A}(\hbar-u_1-u_2)^2&=\frac{1}{2}\left(\frac{(\hbar-\lambda_1-\lambda_2)^2}{\lambda_1(\hbar-\lambda_1-\lambda_2)}+ 
\frac{(\hbar-\lambda_1-\lambda_2)^2}{\lambda_2(\hbar-\lambda_1-\lambda_2)}\right)\\
&=\frac{1}{2}\left(\frac{(\hbar-\lambda_1-\lambda_2)}{\lambda_1}+ 
\frac{(\hbar-\lambda_1-\lambda_2)}{\lambda_2}\right)\\
&=\frac{1}{2}\left(\frac{(\hbar-\lambda_2)}{\lambda_1}+ 
\frac{(\hbar-\lambda_1)}{\lambda_2}-2\right).
\end{align*}
Hence we have:
\begin{align*}
\mathbf{L}^{-1}(\hbar-u_1-u_2)=\frac{1}{2}\left(\frac{(\hbar-\lambda_2)}{\lambda_1}+ 
\frac{(\hbar-\lambda_1)}{\lambda_2}-2\right)(\hbar-u_1-u_2)\\
+\frac{1}{2}\left(\frac{(\hbar-\lambda_2)}{\lambda_1}+ 
\frac{(\hbar-\lambda_1)}{\lambda_2}-2\right)\mathds{1}_{\frac{1}{2}};\\
\mathbf{L}^{-1}(\mathds{1}_{\frac{1}{2}})=\frac{1}{2}(\hbar-u_1-u_2)+\frac{1}{2}\mathds{1}_{\frac{1}{2}}.
\end{align*}
So
\begin{align*}
\mathbf{L}^{-1}(\mathbf{L}(u_1))=\frac{1}{4}\left(\frac{(\hbar-\lambda_2)}{\lambda_1}+ 
\frac{(\hbar-\lambda_1)}{\lambda_2}-2\right)(\hbar-u_1-u_2)\\
+\frac{1}{4}\left(\frac{(\hbar-\lambda_2)}{\lambda_1}+ 
\frac{(\hbar-\lambda_1)}{\lambda_2}-2\right)\mathds{1}_{\frac{1}{2}}
+\frac{1}{4}(\hbar-u_1-u_2)+\frac{1}{4}\mathds{1}_{\frac{1}{2}}.
\end{align*}
It is seen that $\mathbf{L}^{-1}\mathbf{L}$ is not an identity. 
But $\mathbf{L}: H^*_{\CR, \bbT}(X_\A)\to H^*_{\CR, \bbT}(X_\A)$ is injective. 
\end{example}

\subsection{The case of deformation $\MHL$}

Let $\MHL$ be the deformation of $X_\A$ by sub-regular parameter $\lambda$.  We fix a notation that 
$\mathfrak{M}_{\lambda}:=\MHL$.  Then there is a contraction map from 
$\MM_{\lambda}\to \MM_{\lambda}^0$ such that 
$$\MM_{\lambda}\times_{\MM_{\lambda}^0}\MM_{\lambda}\subset \MM_{\lambda}\times \MM_{\lambda}$$
contains components whose dimension are the same as $\dim(\MHL)$. 

Let $(v,\sigma)\in\mbox{Box}$ be an element in the box.  
Then $(X_\A)_{(v,\sigma)}:=X_{\A/\sigma}$ is again a hypertoric Deligne-Mumford stack, associated to the quotient 
stacky hyperplane arrangement $\A/\sigma=(N(\sigma), \beta(\sigma), \theta(\sigma))$. 
Then the twisted sector $(\MM_{\lambda})_{(v,\sigma)}$ of the deformation $\MHL$, associated to 
$(v,\sigma)$,  is also a hypertoric Delinge-Mumford stack, which is the deformation of $X_{\A/\sigma}$. 
Let $\mathbf{Z}:=\MM_{\lambda}\times_{\MM_{\lambda}^0}\MM_{\lambda}$. Then similar to (\ref{Steinberg:correspondence:Chen-Ruan}), the inertia stack $I\mathbf{Z}$ gives the Steinberg correspondence:
\begin{equation}\label{Steinberg:correspondence:Chen-Ruan:deformation}
\mathbf{L}:=L_{I\mathbf{Z}}: H_{\CR,\bbT}^*(\MM_{\lambda})\to H_{\CR,\bbT}^*(\MM_{\lambda}).
\end{equation}

\section{Quantum product by divisors}\label{Section:proof:main:result}

In this section we calculate the small equivariant quantum product by divisors for hypertoric Deligne-Mumford stacks, which proves 
Theorem \ref{Quantum:by:a:divisor}. The quantum product by divisors is reduced to the calculation of 2-point Gromov-Witten invariants.   The computation of quantum product by divisors for smooth hypertoric varieties in \cite{MS} uses a result stated as \cite[Proposition 2.2]{MS}, which was proven in \cite{BMO} in general setting of symplectic resolutions. We prove a similar result  for hypertoric Deligne-Mumford stacks. The arguments here should work for more general {\em stacky symplectic resolutions}.  


\subsection{Reduced virtual fundamental cycle on the deformation} Let $X_\A$ be a hypertoric Deligne-Mumford stack associated with a stacky hyperplane arrangement $\A$.  As in \S \ref{deformation:hypertoricDM:stack}, deformation of $X_\A$ is obtained by varying the level of moment map parameter $\lambda\in (\mathfrak{t}_\cc^{m-d})^*$.  Our algebraic construction $X_\A$ of hypertoric Deligne-Mumford stack corresponds to $\lambda=0$.  Each circuit $S\subset \A$ gives a {\em root hyperplane}
$K_S=\mbox{span}(\iota^*e_i^\vee: i\notin S)\subset (\mathfrak{t}_\cc^{m-d})^*$. By divisor equation, we are interested in the equivariant virtual fundamental cycle 
$[\cM_{0,2}(X_\A, d)]^{\virt}$.  By (\ref{reduced:virtual:class:virtual:class}) we have:
$$[\cM_{0,2}(X_\A, d)]^{\virt}=\hbar\cdot [\cM_{0,2}(X_\A, d)]^{\red}.$$
Let $\phi_0: \aaa^1\to (\mathfrak{t}_\cc^{m-d})^*$ be a generic linear subspace such that it intersects every hyperplane transversely exactly  once at the origin.  The deformation of $X_\A$ over $(\mathfrak{t}_\cc^{m-d})^*$ in \S \ref{deformation:hypertoricDM:stack} restricts to give a smooth map $\mathcal{V}_0\to \aaa^1$ whose fibre over the origin is 
$X_\A$.  The reduced virtual fundamental cycle in \S \ref{reduced:virtual:cycle} has another explanation which was developed in \cite[Proposition 4.1]{BMO}. We stated it as:
$$[\cM_{0,2}(X_\A, d)]^{\red}=[\cM_{0,2}(\mathcal{V}_0, d)]^{\virt}.$$

We deform $\mathcal{V}_0\to \aaa^1$ again. Put $D:=\aaa^1$. Choose a family of maps parametrized by $t\in D$, 
$\phi_{t}: \aaa^1\to (\mathfrak{t}_\cc^{m-d})^*$, such that for $t=0$ we have the $\phi_0$ above, and for $t\neq 0$ sufficiently close to $0$, the image of $\phi_t$ intersects each hyperplane transversely at distinct points. 
Write $\mathcal{V}_t\to \aaa^1$ for the family obtained by restriction to $\phi_t$.  The deformation invariance implies that 
$$[\cM_{0,2}(\mathcal{V}_0, d)]^{\virt}=[\cM_{0,2}(\mathcal{V}_t, d)]^{\virt}.$$
By construction, compact curves can only be in the fibres of $\mathcal{V}_t\to \aaa^1$ over points of intersects with root hyperplanes. We conclude that only multiples of $\beta_S$ appears in the quantum corrections. To study the quantum correction from class $m\beta_S$, we study $\MHL$ where $\lambda$ is subregular (i.e., lying on $K_S$ and only $K_S$). 
Proposition \ref{deformation:hypertoric:subregular} implies that compact curves lie in the fibres of $\mathcal{M}^{S}\to \mathcal{M}_0^{S}$.  Since $\mathcal{M}^{S}$ is symplectic and the fibres are isotropic, its normal bundle is identified with the cotangent bundle.  So this reduces the calculation of equivariant Gromov-Witten invariants to the case of cotangent bundle $T^*\pp^{n}_{\mathbf{w}}$ of  $\pp^{n}_{\mathbf{w}}$.

\subsection{Broken and unbroken twisted maps}\label{broken:unbroken}

We list some properties of broken and unbroken maps as in \cite[Section 3.8.2]{OP},  see also \cite[Section 7.3]{MO}.  Consider 
the moduli stack $\cM_{0,k}(X_\A, d)$ for $d>0$ with the torus $\mathbb{T}:=T\times\cc^\times$-action.  In each $\bbT$-fixed  twisted stable map 
$f: \mathcal{C}\to X_\A$ in $\cM_{0,k}(X_\A, d)$, the twisted curve $\mathcal{C}$ is a chain of twisted rational curves
$$\cC_1\cup\cdots\cup\cC_k$$ with $k$ markings. 
Let $p\in \cC$ be a marking, the $T\times\cc^\times$ weight $\omega_p$ at $p$ is defined by the 
$T\times\cc^\times$-representation of the tangent space to $\cC$ at $p$.  Similarly if $\mathcal{P}\subset \cC$ is a component 
incident to a twisted node $s$. A   $T\times\cc^\times$ weight $\omega_{\mathcal{P},s}$ at $s$ is defined by the 
$T\times\cc^\times$-representation of the tangent space to $\mathcal{P}$ at $s$.

If at a node $s$, the $\bbT$ weights of the two branches are opposite and nonzero then we say that $f$ is an {\em unbroken chain}.  Moreover we say that $f$ is an {\em unbroken twisted map} if it satisfies one of the following three conditions:
\begin{enumerate}
\item $f$ comes from a twisted map $f:  \cC\to X_\A^{\mathbb{T}}$,
\item $f$ is an unbroken chain,
\item the curve $\cC$ is a chain of twisted rational curves
$$\cC_0\cup\cC_1\cup\cdots\cup\cC_k$$
such that $\cC_0$ is contracted by $f$, the marked points lie in $\cC_0$ and the remaining of $\cC$ forms a unbroken chain. 
\end{enumerate}
The $\bbT$-fixed twisted stable maps which do not satisfy the above conditions are called {\em broken twisted maps}.

\subsection{Calculation on $T^*\pp^{n}_{\mathbf{w}}$}\label{section:case:weighted:projective:stack}

The cotangent bundle $X:=T^*\pp^{n}_{\mathbf{w}}$ is a hypertoric Deligne-Mumford stack.  
Recall that in \S \ref{Section:Steinberg:correspondence} the components of $I\mathbf{Z}\subset IX\times IX$ define the Steinberg correspondence (\ref{Steinberg:correspondence:Chen-Ruan}) for equivariant Chen-Ruan cohomology 
$H^*_{\CR, \bbT}(X)$.
If $\{\gamma_i\}$ is a basic for the Chen-Ruan cohomology  $H^*_{\CR, \bbT}(X)$, we let 
$\{\gamma^i\}$ the dual basis under orbifold Poincar\'e pairing. 
Let $D$ be a divisor class in 
$H_{\bbT}^{2}(X,\qq)$.  
In this case all curves lie in the weighted projective stack, hence 
$\NE(X)_{\zz}=\zz_{\geq 0}$. 
Let $\ell\subset \NE(X)_{\zz}$ be the primitive line. 
Recall that 
the quantum product is defined by:
$$( D\star \gamma_i, \gamma_j)=\sum_{d\ell\geq 0}Q^{d\ell}\langle D, \gamma_i, \gamma_j\rangle^{X}_{0,3,d}
= \sum_{d\ell\geq 0} (\int_{d\ell}D) Q^{d\ell}\langle \gamma_i, \gamma_j\rangle^{X}_{0,2,d}$$
Let 
$$\mathbf{ev}:=\ev_1\times \ev_2: \cM_{0,2}(X,d)\to IX\times IX$$
be the evaluation map. Then using the relationship:
$$[\cM_{0,2}(X, d)]^{\virt}=\hbar\cdot [\cM_{0,2}(X, d)]^{\red}$$ for $d\neq 0$, 
we can write the quantum product by $D$ as:
\begin{align}\label{quantum:D}
D\star \gamma_i&=D\cup_{\CR} \gamma_i+
 \sum_{d\ell> 0}\int_{d\ell}D\cdot Q^{d\ell}\langle \gamma_i, \gamma_j\rangle^{X}_{0,2,d}\gamma^{j}\\
 &=D\cup_{\CR} \gamma_i+
 \sum_{d\ell> 0}\int_{d\ell}D\cdot Q^{d\ell}\cdot \hbar \cdot \inv^*\cdot Ip_{2\star}(Ip_1^\star(\gamma_i)\cap \mathbf{ev}_{\star}[\cM_{0,2}(X, d)]^{\red}). \nonumber
\end{align}
We process  (\ref{quantum:D}) by localization of the $\bbT$-action on the moduli stack. 

Recall that in \S \ref{sec:cotangentbundle:Steinberg},   the twisted sectors of the weighted projective stack 
$\pp^n_{\mathbf{w}}$ are indexed by the set 
$$F:=\{\frac{a}{w_i}| 0\leq a< w_i, 0\leq i\leq n\},$$
and the twisted sector corresponding to $f=\frac{a}{w_i}$ is 
the sub-weighted projective stack $\pp([\mathbf{w}:d])$, where $d=|e^{2\pi i f}|$ is the order of $e^{2\pi i f}$. 
The components of the inertia stack $IX$ of the cotangent bundle $X=T^*\pp^n_{\mathbf{w}}$ are also indexed by the set 
$F$, and the twisted sector $X_{f}$ corresponding to $f=\frac{a}{w_i}$ is 
the cotangent bundle $T^*\pp([\mathbf{w}:d])$ of the sub-weighted projective stack $\pp([\mathbf{w}:d])$.

First we have the following fact for the pushforward of virtual fundamental cycle:
\begin{lem}\label{support:lemma}
$$\mathbf{ev}_{\star}[\cM_{0,2}(X, d)]^{\red}\in H^{2n}_{\CR, \bbT}(IX\times IX),$$
where $2n=\dim(X)$. 
\end{lem}
\begin{proof}
Let $\cM_{0,2}(X, d)^{f_1, f_2}$ be the component of the moduli stack indexed by the elements 
$f_1, f_2\in F$ such that the evaluation maps 
$\ev_1, \ev_2$ have images in $X_{f_1}$ and $X_{f_2}$, respectively. 
The cycle $[\cM_{0,2}(X, d)^{f_1,f_2}]^{\red}$ has dimension 
\begin{align*}
\dim([\cM_{0,2}(X, d)^{f_1,f_2}]^{\red})&=\dim(X)-\age(X_{f_1})-\age(X_{f_2})\\
&=n-\age(X_{f_1})+n-\age(X_{f_2})\\
&=\dim(\pp([\mathbf{w}: d_1]))+\dim(\pp([\mathbf{w}: d_2])),
\end{align*}
where $d_1, d_2$ are the orders of $e^{2\pi i f_1}, e^{2\pi i f_2}$. 
Note that 
$$\mathbf{ev}_{\star}[\cM_{0,2}(X, d)]^{\red}\in H^{2n-\tiny\age(X_{f_1})-\tiny\age(X_{f_2})}_{\bbT}(X_{f_1}\times X_{f_2}).$$ 
Hence 
if $f_1, f_2$ belong to the same local group and have the same order, then $X_{f_1}\cong X_{f_2}$ and  $\mathbf{ev}_{\star}[\cM_{0,2}(X, d)^{f_1,f_2}]^{\red}$ supports on the diagonal $X_{f_1}\times X_{f_2}$ or 
$\pp([\mathbf{w}: d_1]))\times\pp([\mathbf{w}: d_2])$.  If $f_1, f_2$ do not have the same orders, then 
 $\mathbf{ev}_{\star}[\cM_{0,2}(X, d)^{f_1,f_2}]^{\red}$ supports on 
$\pp([\mathbf{w}: d_1]))\times\pp([\mathbf{w}: d_2])$.
\end{proof}

Lemma \ref{support:lemma} tells us that the pushforward of reduced cycle supports on the Lagrangian cycles inside 
$IX\times IX$.  Next we use localization of reduced virtual fundamental cycles as in \cite{GP}, \cite{MO} to calculate the sign of the support. 
We first have the following lemma:
\begin{lem}\label{preliminary:lemma}
Let $\pp^{1}_{s_1,s_2}$ be a $\pp^1$-orbifold with stacky points 
$P_1=[1,0]=B\mu_{s_1}$ and $P_2=[0,1]=B\mu_{s_2}$.  Assume that $\mathcal{V}$ is a vector bundle on 
$\pp^{1}_{s_1,s_2}$.  Let $T$ be a torus acting on $\mathcal{V}$ with no zero weights. Then we have the equivariant Euler class
$$e_{T}(H^0(\mathcal{V}\oplus\mathcal{V}^*-H^1(\mathcal{V}\oplus\mathcal{V}^*))=(-1)^{\chi(\mathcal{V})+\rk(H^1(\mathcal{V}\oplus\mathcal{V}^*)^{T}}e_{T}(\mathcal{V}_{P_1})^{\inv}\cdot e_{T}(\mathcal{V}_{P_2})^{\inv},$$
where $(\mathcal{V}_{P_i})^{\inv}$ is the invariant part of the restriction of $\mathcal{V}$ to $P_i$. 
\end{lem}
\begin{proof}
By a result of Martens-Thaddeus \cite{MT}, the vector bundle $\mathcal{V}$ splits as a direct sum of line bundles.  Then the formula comes from \cite[Example 98]{MLiu}.
\end{proof}

The torus $\bbT$ acts on the moduli stack $\cM_{0,2}(X,d)$.  The following lemma is very similar to Lemma 6 in \cite{OP} (with the same proof):
\begin{lem}
The $\bbT$-fixed broken twisted maps contribute zero under $\bbT$-localization. 
\end{lem}

The $\bbT$-fixed points of $X$ are all contained in the $0$-section $\pp^n_{\mathbf{w}}$, and are the same as the fixed points set of the torus action on $\pp^n_{\mathbf{w}}$.
The torus $\bbT$-fixed one-dimensional orbits of $X$ are contained in $\pp^n_{\mathbf{w}}$ and gerbes of weighted projective lines.  The following is a stacky version of \cite[Lemma 7]{OP}:

\begin{lem}\label{broken:maps:reducible}
There are no unbroken twisted stable maps in $\cM_{0,2}(X,d)^{\bbT}$ with reducible domains connecting two $\bbT$-fixed points in $\pp^n_{\mathbf{w}}$.
\end{lem}
\begin{proof}
The torus fixed one-dimensional orbits of $X$ are gerbes of weighted projective lines. Then the result follows from the fact that there are no $\bbT$-fixed unbroken twisted maps connecting the $\bbT$-fixed points of $X$ to itself. 
\end{proof}

Look at the evaluation map 
$\mathbf{ev}: \cM_{0,2}(X,d): \to IX\times IX$,  from Lemma \ref{broken:maps:reducible},  the cycle 
$\mathbf{ev}_{\star}([\cM_{0,2}(X,d)]^{\red})$ is supported away from the diagonal of $IX\times IX$. 
Also the affinization $X\to \overline{X}^0$ contracts the zero section. Hence 
the cycle $\mathbf{ev}_{\star}([\cM_{0,2}(X,d)]^{\red})$ factors through 
$IX_{I\overline{X}^0}IX\subset IX\times IX$. 
Lemma \ref{broken:maps:reducible} tells us that the only $\bbT$-fixed twisted stable maps 
come from a cyclic cover map to a $\bbT$-fixed one dimensional orbit inside weighted projective stack 
$\pp^n_{\mathbf{w}}$. 
So let 
$\pp^1_{s_1, s_2}\to X$ be a $\bbT$-fixed twisted stable map to $X$, such that 
it maps to the $\bbT$-fixed one dimensional 
orbit inside $\pp^n_{\mathbf{w}}$ corresponding to $f_1, f_2\in F$. 
We apply the formula in Lemma \ref{preliminary:lemma} to 
$\mathcal{V}=f^\star T_{\pp^n_{\mathbf{w}}}$. 
Note that 
$f^\star T_{X}=\mathcal{V}\oplus \mathcal{V}^*$. 
Let $P:=\pp^n_{\mathbf{w}}$. 
By localization as in \cite[Chapter 11]{MO},  the contribution of $f$ to 
$\mathbf{ev}_{\star}([\cM_{0,2}(X,d\ell)]^{\red})$ is :
$$\frac{1}{\mbox{Aut}(f)}e_{\bbT}(H^{0}-H^{1}(\mathcal{V}\oplus\mathcal{V}^*))^{-1}$$
and 
the component $\mathbf{ev}_{\star}([\cM_{0,2}(X,d)^{f_1,f_2}]^{\red})\in H_*(X_{f_1}\times X_{f_2})$ is:
$$\frac{1}{d}(-1)^{\rk(\mathcal{V})+\int_{d\ell}c_1(T_{\pp^n_{\mathbf{w}}})}(-1)^{-\age(P_{f_1})}[P_{f_1}]\times 
(-1)^{-\age(P_{f_2})}[P_{f_2}].$$
Set 
$$\Gamma_{f_1, f_2}:=(-1)^{-\age(P_{f_1})}[P_{f_1}]\times 
(-1)^{-\age(P_{f_2})}[P_{f_2}]; \quad  \Gamma:=\bigoplus_{f_1, f_2\in F}\Gamma_{f_1, f_2}.$$

For fixed $f_1, f_2\in F$, it is not hard to see from Proposition \ref{classification:P1orbifold} that 
those degrees $d\in \zz_{\geq 0}$ such that there exist maps 
$\pp^1_{s_1, s_2}\to P$ with stack structures specified by 
$(f_1, f_2)$, are exactly of the following form\footnote{The existence of $\gamma(f_1, f_2)$ imposes constraints on the possible pairs $(f_1, f_2)$.}:
$$d=\delta\cdot \lcm(w_0, \cdots,w_n)+ r(f_1, f_2),$$
where $\delta\in \zz_{\geq 0}$, 
$r(f_1, f_2)\in \{0, 1,\cdots, \lcm(w_i)-1\}$ is a number such that 
$\langle \frac{r(f_1, f_2)}{w_i}\rangle=f_1$, $\langle \frac{r(f_1, f_2)}{w_j}\rangle=f_2$, and
$r(f_1, f_2)\equiv 0~ \mbox{mod}~ w_k$ for $w_k\neq   w_i, w_j$ 
(here $f_1=\frac{a_1}{w_i}$, $f_2=\frac{a_2}{w_j}$).
From (\ref{quantum:D}),  the quantum product by divisor $D$ is: 
\begin{align}\label{quantum:D:formula}
D\star \gamma_i &=D\cup_{\CR} \gamma_i+
 \sum_{d\ell> 0}\int_{d\ell}D\cdot Q^{d\ell}\cdot \hbar \cdot \inv^*\cdot Ip_{2\star}(Ip_1^\star(\gamma_i)\cap \mathbf{ev}_{\star}[\cM_{0,2}(X, d)]^{\red}). \nonumber \\
 &=D\cup_{\CR} \gamma_i+
\int_{\ell}D\cdot  \sum_{d\ell> 0} Q^{d\ell}\cdot d\cdot \hbar \cdot \inv^*\cdot Ip_{2\star}(Ip_1^\star(\gamma_i)\cap \mathbf{ev}_{\star}[\cM_{0,2}(X, d)]^{\red}). \nonumber \\
&=D\cup_{\CR} \gamma_i+
\hbar\int_{\ell}D\cdot (-1)^n\cdot \sum_{\delta\geq 0} \sum_{(f_1,f_2)\in F^2}Q^{(\delta\cdot\tiny\lcm(w_i)+r(f_1, f_2))\ell} \\ \nonumber
&\cdot (-1)^{(\delta\cdot\tiny\lcm(w_i)+r(f_1, f_2))\int_{\ell}c_1(T_{P})}\cdot \inv^*\cdot Ip_{2\star}(Ip_1^\star(\gamma_i)\cap \Gamma_{f_1,f_2}). \\ \nonumber 
&=
D\cup_{\CR} \gamma_i+
\begin{cases}
\hbar\int_{\ell}D\cdot 
(-1)^n\frac{(Q^{\ell}(-1)^{\sum_i\frac{1}{w_i}})^{r(f_1,f_2)}}{1-((-1)^{\sum_i\frac{1}{w_i}}Q^{\ell})^{\tiny\lcm(w_i)}}
\Gamma_{f_1,f_2}; &  r(f_1,f_2)\neq 0;\\
 \hbar\int_{\ell}D\cdot 
(-1)^n\frac{(Q^{\ell}(-1)^{\sum_i\frac{1}{w_i}})^{\tiny\lcm(w_i)}}{1-((-1)^{\sum_i\frac{1}{w_i}}Q^{\ell})^{\tiny\lcm(w_i)}}
\Gamma_{f_1,f_2}; &  r(f_1,f_2)= 0,
\end{cases}
\end{align}
where in the calculation we use $\int_{\ell}c_1(T_{P})=\sum_i\frac{1}{w_i}$, and $\Gamma_{f_1, f_2}$ is the Steinberg correspondence in (\ref{Steinberg:correspondence:Chen-Ruan}). 

\subsection{Proof of Theorem \ref{Quantum:by:a:divisor}}
 Recall that in \S \ref{Section:deformation:hypertoricDM}, the hypertoric Deligne-Mumford stack $X_\A$ can be deformed into $X_\lambda$ for sub-regular parameter.  Proposition \ref{deformation:hypertoric:subregular} says that $X_\lambda$ contains a substack $\cM^S\to \cM^S_0$ for any circuit $S\subset \A$, which is the weighted projective bundle over a singular base $\cM^S_0$.  The fibre is the weighted projective stack 
$\pp^{|S|-1}_{\mathbf{w}}$ with fibre-wise normal bundle the cotangent bundle $T^*\pp^{|S|-1}_{\mathbf{w}}$.  Hence the curve class in the $2$-point Gromov-Witten invariants must lie in $\pp^{|S|-1}_{\mathbf{w}}$.  The twisted sectors $(X_\lambda)_{f_i}$ for $f_i$ is determined by $f_i\in F$ for $i=1,2$.  In this case 
the Lagrangian cycle 
$\Gamma_{f_1,f_2}\subset IX_\lambda\times IX_\lambda$ is actually  lying in 
$I\cM^S\times I\cM^S$.  Then same calculation as in \S \ref{section:case:weighted:projective:stack} concludes that  we have the same as (\ref{quantum:D:formula}). The theorem is proved.

\section{Gromov-Witten theory of hypertoric Deligne-Mumford stacks}\label{Section:GW:theory:hypertoricDM}

In this section we compare the $\bbT$-equivariant Gromov-Witten theory of $X_\A$ with the $\bbT$-equivariant Gromov-Witten theory of the associated Lawrence toric Deligne-Mumford stack 
$X_{\mathbf{\theta}}$.   

\subsection{Comparison of Gromov-Witten invariants}

Let $\A$ be a stacky hyperplane arrangement and $X_\A$ the corresponding hypertoric Deligne-Mumford stack.
As discussed in \S \ref{Section:hypertoricDMstack}, $X_\A$ is defined as a closed substack of the corresponding Lawrence toric Deligne-Mumford stack $X_{\mathbf{\theta}}$.    The stacky fan $\mathbf{\Sigma_{\theta}}$ determines an open sub-variety $X:=\cc^{2m}\setminus V(\mathcal{I}_{\theta})\subset \cc^{2m}$ and the Lawrence toric Deligne-Mumford stack  $X_{\mathbf{\theta}}$ is the quotient stack $[X/G]$.  Let $Y\subset X$ be the closed sub-variety determined by the ideals in (\ref{ideal1}). The hypertoric Deligne-Mumford stack $X_\A$ is the quotient stack $[Y/G]$. 

\begin{lem}\label{trivial:N}
The normal bundle $N:=N_{X_\A/X_{\mathbf{\theta}}}$ of $X_\A\subset X_\theta$ is trivial. 
\end{lem}
\begin{proof}
This can be deduced from generalized Euler sequence in \cite[Section 1.1.1]{Edidin}.  Alternatively, $N$ is trivial because in the equations for $X_\A$ in (\ref{ideal1}), $z_i$ and $w_i$ are sections of line bundles of $X_\theta$ that are dual to each other.  
\end{proof}

The tori $T$ and $\bbT$ act on both $X_\A$ and $X_{\mathbf{\theta}}$. The inclusion $\iota: X_\A\hookrightarrow X_{\mathbf{\theta}}$ is equivariant with respect the the actions of these two tori.

\begin{lem}\label{fixed:loci}
The fixed loci $X_\A^{T}$ is the same as the fixed loci $X_{\mathbf{\theta}}^{T}$.
\end{lem}
\begin{proof}
It is clear that $X_\A^{T}\subset X_{\mathbf{\theta}}^{T}$.  Let $T^\prime$ be the one dimensional torus defined by the vector  $\sum_{i=1}^{m}b_i\in N$.  The  $T^\prime$-action on the Lawrence toric Deligne-Mumford stack is induced by the multiplication by non-zero complex numbers on $\cc^{2m}$.  This is the action defined in  \cite[Lemma 6.5]{HS}. 
It is shown in \cite{HS} and \cite{Edidin} that $X_\A^{T^\prime}= X_{\mathbf{\theta}}^{T^\prime}$. Thus 
$ X_{\mathbf{\theta}}^{T} \subset X_{\mathbf{\theta}}^{T^\prime}=X_\A^{T^\prime}= X_\A^{T}$, 
which implies that $ X_{\mathbf{\theta}}^{T} \subset X_\A^{T}$.
\end{proof}

\begin{lem}\label{one-dimensional:orbits}
The $T$-actions on both $X_\A$ and $X_{\mathbf{\theta}}$ have identical compact one-dimensional orbits. 
\end{lem}
\begin{proof}
The projective substack of $X_\A$ and $X_{\mathbf{\theta}}$ is the core $C(X_\A)$ defined in \S \ref{core:hypertoricDMstack}.  Since the core $C(X_\A)$ and any compact curve are contracted by the affinization map 
$X_{\mathbf{\theta}}\to \mbox{Spec} H^0(\mathcal{O}_{X_{\mathbf{\theta}}})$, we see that the compact one-dimensional $T$-orbits are all contained in $C(X_\A)$. 
\end{proof}

 The inclusion $\iota$ induces an inclusion of moduli stacks $$\iota: \cM_{g,n}(X_\A,d)\hookrightarrow  \cM_{g,n}(X_{\mathbf{\theta}},d),$$
which is $\bbT$-equivariant.  

By \cite{Edidin}, the pullback
$$\iota^\star: H_{\CR, \bbT}^*(X_{\mathbf{\theta}})\to H^*_{\CR, \bbT}(X_\A)$$
is a ring isomorphism. 

The following result concerns genus $0$ Gromov-Witten invariants.

\begin{prop}\label{invariants:compare}
Let $\gamma_1,\cdots,\gamma_n\in H^*_{\CR, \bbT}(X_{\mathbf{\theta}})$, and 
$k_1,\cdots,k_n\in\zz_{\geq 0}$. Then there is an equality on descendant $\bbT$-equivariant Gromov-Witten invariants
$$\left\langle \prod_{i=1}^{n}\gamma_i\psi_i^{k_i}\right\rangle_{0,n,d}^{X_{\mathbf{\theta}}}=
\frac{\left\langle \prod_{i=1}^{n}\iota^\star(\gamma_i)\psi_i^{k_i}\right\rangle_{0,n,d}^{X_\A}}{e_{\bbT}(N)},$$
where $e_{\bbT}(N)$ is the equivariant Euler class of $N$. 
Note that the Gromov-Witten invariants take values in $H^*_{\bbT}(pt)$. 
\end{prop}
\begin{proof}
We compare the two invariants by 
$\bbT$-equivariant localization. 
By Lemma \ref{fixed:loci} and \ref{one-dimensional:orbits}, we have 
$\cM_{g,n}(X_\A,d)^{\bbT}=\cM_{g,n}(X_{\mathbf{\theta}},d)^{\bbT}$ with identical $\bbT$-fixed obstruction theories. 
A simple check shows that $e_\bbT$ of the virtual normal bundle for $X_{\mathbf{\theta}}$ and $X_\A$ differ by 
$e_\bbT(H^1(f^\star N)-H^0(f^\star N))$, where $f$ is the universal twisted stable map. 
Since $N$ is trivial (Lemma \ref{trivial:N}), we have 
$H^1(f^\star N)=0$ and $e_\bbT(H^0)$ is identified with $e_\bbT(N)$. The result follows. 
\end{proof}

\begin{cor}\label{quantum:ring:isomorphism}
The pullback $\iota^\star$ gives an isomorphism 
$$QH^*_{\bbT,\bbig}(X_{\mathbf{\theta}})\cong QH^*_{\bbT,\bbig}(X_\A).$$
\end{cor}
\begin{proof}
For $\gamma_1,\gamma_2\in H_{\CR,\bbT}^*(X_{\mathbf{\theta}})$, it suffices to show that 
\begin{equation}\label{key:equality}
\iota^\star(\gamma_1\star_{t}\gamma_2)=\iota^\star\gamma_1\star_{\iota^\star t}\iota^\star\gamma_2
\end{equation}
for $t\in H^*_{\CR, \bbT}(X_{\mathbf{\theta}})$. 
Given $d\in H_2(X_{\mathbf{\theta}})$, the contribution from $d$ to the LHS of (\ref{key:equality}) is the sum over $n$ of 
$$\left\langle\gamma_1, \underbrace{t, \cdots, t}_{n}, \gamma_2,\gamma\right\rangle_{0,n+3,d}^{X_{\mathbf{\theta}}}
\mbox{PD}_{X_{\mathbf{\theta}}}(\gamma),$$
where $\mbox{PD}_{X_{\mathbf{\theta}}}(\gamma)$ is the orbifold Poincare dual of $\gamma$. 
By Proposition \ref{invariants:compare}, this is equal to 
$$\left\langle\iota^\star\gamma_1, \underbrace{\iota^\star t, \cdots, \iota^\star t}_{n}, \iota^\star\gamma_2, \iota^\star\gamma\right\rangle_{0,n+3,d}^{X_{\A}}
\frac{\mbox{PD}_{X_{\mathbf{\theta}}}(\gamma)}{e_\bbT(N)}.$$
By localization, we have the following relation for orbifold Poincar\'e pairings:
$$\int_{IX_{\mathbf{\theta}}}a\cup b
=\int_{IX_{\mathbf{\theta}}^{\bbT}}\frac{a\cup b}{e_{\bbT}(N_{IX_{\mathbf{\theta}}^{\bbT}/IX_{\mathbf{\theta}}})}
=\int_{IX_{\A}^{\bbT}}\frac{\iota^\star a\cup \iota^\star b}{e_{\bbT}(N_{IX_{\A}^{\bbT}/IX_{\A}})\cdot e_\bbT(N)}
=\int_{IX_{\A}}\frac{\iota^\star a\cup \iota^\star b}{e_\bbT(N)},$$
where we use \cite[Proposition 3.9]{Edidin}, which identifies $N_{IX_\A/IX_\theta}$ with the pullback of $N$ to $IX_\A$ under $IX_\A\to X_\A$. 
Therefore 
$\mbox{PD}_{X_\A}(\iota^\star\gamma)=\iota^\star\frac{\mbox{PD}_{X_{\mathbf{\theta}}}(\gamma)}{e_T(N)}$ and we get
\begin{align*}
\left\langle\gamma_1, \underbrace{t, \cdots, t}_{n}, \gamma_2,\gamma\right\rangle_{0,n+3,d}^{X_{\mathbf{\theta}}}
\iota^\star\mbox{PD}_{X_{\mathbf{\theta}}}(\gamma)&=\left\langle\iota^\star\gamma_1, \underbrace{\iota^\star t, \cdots, \iota^\star t}_{n}, \iota^\star\gamma_2, \iota^\star\gamma\right\rangle_{0,n+3,d}^{X_{\A}}\iota^\star \frac{\mbox{PD}_{X_{\mathbf{\theta}}}(\gamma)}{e_T(N)}\\
&=\left\langle\iota^\star\gamma_1, \underbrace{\iota^\star t, \cdots, \iota^\star t}_{n}, \iota^\star\gamma_2, \iota^\star\gamma\right\rangle_{0,n+3,d}^{X_{\A}}\mbox{PD}_{X_\A}(\iota^\star\gamma)
\end{align*}
which is the degree $d$ contribution to the RHS of (\ref{key:equality}). 
\end{proof}

\begin{rmk}[Higher genus Gromov-Witten invariants]

We briefly discuss what happens in higher genus.  Results in toric Gromov-Witten theory \cite{CCIT} shows that $QH_{\bbT, \bbig}(X_{\mathbf{\theta}})$ is semi-simple.  By Corollary \ref{quantum:ring:isomorphism}, the same is true for 
$QH_{\bbT, \bbig}(X_{\A})$. The Givental-Teleman reconstruction \cite{Givental}, \cite{Teleman} applies to determine higher genus Gromov-Witten invariants of $X_{\mathbf{\theta}}$ and $X_\A$, provided the $R$-calibrations can be specified. 

Since equivariant Gromov-Witten theory is not conformal, we follow the toric case to specify $R$ by computing its degree $0$ limit.  The argument for Proposition \ref{invariants:compare} applies to show that the degree $0$, genus $g$ Gromov-Witten invariants of  $X_{\mathbf{\theta}}$ are equal to degree $0$, genus $g$ Gromov-Witten invariants of  $X_{\A}$ twisted by $(N, e_\bbT(-)^{-1})$, which amounts to include Hodge classes. Since $N$ is trivial, we may apply results of \cite{FP} to remove the Hodge classes.  Writing in Givental's formalism \cite{Givental}, this leads to $R_{\deg=0}^{X_{\mathbf{\theta}}}=\Delta R_{\deg=0}^{X_\A}$, where $\Delta$ is an operator that can be explicitly written down.  Together with ancestor/descendant relation, we we should obtain an equality of descendant potentials: $D_{X_{\mathbf{\theta}}}=\hat{\Delta}D_{X_\A}$. We leave the details to the interested readers.
\end{rmk}

 
\subsection{Proof of Theorem \ref{main2:small:quantum:ring}}\label{sec:pf_ring_presentation}
By Corollary \ref{quantum:ring:isomorphism}, the small quantum ring $QH^*_{\bbT}(X_\A)$ is isomorphic to the small quantum ring $QH^*_{\bbT}(X_\theta)$ of the associated Lawrence toric Deligne-Mumford stack $X_\theta$. Theorem \ref{main2:small:quantum:ring} is obtained by deriving a presentation of $QH^*_{\bbT}(X_\theta)$ using known results in toric Gromov-Witten theory.

Recall the construction of Lawrence toric Deligne-Mumford stack $X_\theta$ in \S \ref{LawrenceDMstack}, the toric fan $\Sigma_\theta$ is constructed  from the Gale dual of the map 
$$ \mathbb{Z}^{m}\oplus \mathbb{Z}^{m}\stackrel{(\beta^{\vee},-\beta^{\vee})}{\longrightarrow}
DG(\beta).$$ 
Let 
$D_1,\cdots,D_m,D_1^\prime,\cdots,D_{m}^\prime$ be the images of the above map, which are the toric divisors of 
$X_\theta$. Then the Chern class of the tangent bundle if given by $\sum_{i=1}^m(D_i+D_i^\prime)=0$, hence the toric Deligne-Mumford stack $X_\theta$ is Calabi-Yau. As explained in \cite{Ir}, it follows that the toric mirror theorem of \cite{CCIT} implies that the $\bbT$-equivariant quantum cohomology ring of $X_\theta$ with quantum parameters in $H^{\leq 2}_{orb}$ is isomorphic to Batyrev's quantum ring, taking into account of the {\em mirror map} along $H^{\leq 2}_{orb}$. A direct computation shows that when the quantum parameters are restricted to $H^{\leq 2}\subset H^{\leq 2}_{orb}$, Batyrev's quantum ring is isomorphic to the ring in Theorem \ref{main2:small:quantum:ring}. Therefore it remains to show that the mirror map is the identity when restricted to $H^{\leq 2}$. 

Recall that the Gale dual map is 
$$
\beta_{L}: \mathbb{Z}^{2m}\rightarrow N_{L},
$$
where $\overline{N}_{L}$
is a lattice of dimension $2m-(m-d)$. The map $\beta_{L}$ is given by the integral vectors  $\{b_{L,1},\cdots,
b_{L,m},b'_{L,1},\cdots,
b^{'}_{L,m}\}$ and $\beta_{L}$ is called  the Lawrence lifting of $\beta$.
From Remark 2.3 in \cite{JT}, 
$\{b_{L,1},\cdots,
b_{L,m},b'_{L,1},\cdots,
b^{'}_{L,m}\}$ are the vectors $$\{(b_{1},e_{1}),\cdots,
(b_{m},e_{m}),(0,e_{1}),\cdots,
(0,e_{m})\},$$ where $\{e_{i}\}$ are the standard bases of $\mathbb{Z}^{m}$.
The toric divisors $D_i, D_i^\prime$ define equivariant cohomology classes $u_i, u_i^\prime\in H^2(X_\theta)$, and we have 
$u_i^\prime=\hbar-u_i$. 

We only need to prove that in this case the mirror map is trivial. 

The mirror map restricted to $H^{\leq 2}$ is calculated in \cite[Section 6.3]{CCLT}. To study the mirror map, we need to understand elements of the set 
$\mathbb{K}_{\mbox{\tiny eff}}$, which is the nef cone.  By definition an element in $\mathbb{K}_{\mbox{\tiny eff}}$
is of the form 
$$\beta=\sum_{i=1}^{m}c_ie_i+ \sum_{i=1}^{m}c^\prime_ie^\prime_i\in \zz^m\oplus\zz^m$$
such that 
$$\sum_{i}c_i b_{L,i}+\sum_{i}c_i^\prime b_{L,i}^\prime=0,$$
i.e.
$$\sum_{i}c_i b_i=0, ~\mbox{and}~ c_i+c_i^\prime=0 ~\mbox{for any}~ i.$$
Moreover, $c_i=\langle D_i, \beta\rangle$, $c_i^\prime=\langle D^\prime_i, \beta\rangle$. 
If $\beta\in \Omega_i^{X_\theta}$, where  $\Omega_i^{X_\theta}$ is defined on page 39 of \cite{CCLT}, 
then we have 
$\langle D_i, \beta\rangle<0$, $\langle D_j, \beta\rangle\geq 0$ for $D_j\neq D_i$. 
For $j\neq i$, we have $\langle D_j, \beta\rangle+\langle D^\prime_j, \beta\rangle=0$. Since both are nongenative, we have 
$\langle D_j, \beta\rangle=\langle D^\prime_j, \beta\rangle=0$ for $j\neq i$. 
On the other hand, the definition of $\beta$ implies that 
$$\sum_{k}\langle D_k, \beta\rangle b_k=0\in N.$$
Because $\langle D_j, \beta\rangle=0$ for $j\neq i$, this reduces to 
$\langle D_i, \beta\rangle b_i=0$, which contradicts the requirement that 
$\langle D_i, \beta\rangle<0$. So 
$\Omega_{i}^{X_\theta}=\emptyset$.  Similarly the set corresponding to 
$D_i^\prime$ is empty.  We thus conclude that the mirror map restricted to $H^{\leq 2}$ is trivial. This completes the proof of Theorem \ref{main2:small:quantum:ring}.

We end with calculating the small equivariant quantum cohomology rings for  two examples.  In the second example we also explain how to get the quantum Stanley-Reisner relation from the quantum product formula in Theorem \ref{Quantum:by:a:divisor}. 
\begin{example}
Let $\A=(N, \beta, \theta)$ be the stacky hyperplane arrangement given by
$N=\zz$, $\beta: \zz^2\to N$ is given by $\{-1, 1\}$, and $\theta=1$ in $DG(\beta)=\zz$. 
The hypertoric Deligne-Mumford stack is $T^*\pp^1$, which is a closed substack of the Lawrence toric Deligne-Mumford stack 
$X_{\mathbf{\theta}}=\oO_{\pp^1}(-1)\oplus\oO_{\pp^1}(-1)$.  
Let $u_1, u_2$ be the hyperplane classes of $\pp^1$, and $\hbar$ the first Chern class of the extra $\cc^\times$-action. 
Our formula (\ref{Quantum:by:a:divisor}) is:
$$u_1\star u_2=\frac{Q^\ell}{1-Q^{\ell}}(\hbar-u_1)\cdot (\hbar-u_2).$$
Then using this formula we can find that the quantum ring 
$$QH_{\bbT}^{*}(X_\A)=\frac{\qq[u_1, u_2, \hbar]}{(u_1\star u_2-Q^{\ell}(\hbar-u_1)\star(\hbar-u_2))}.$$
This was computed in \cite{MS}.
\end{example}

\begin{example}
Let $\A=(N, \beta, \theta)$ be the stacky hyperplane arrangement given by
$N=\zz$, $\beta: \zz^2\to N$ is given by $\{-2, 1\}$, and $\theta=1$ in $DG(\beta)=\zz$. 
The hypertoric Deligne-Mumford stack in this case is $X_\A=T^*\pp(1,2)$, which is a closed substack of the Lawrence toric Deligne-Mumford stack 
$X_{\mathbf{\theta}}=\oO_{\pp(1,2)}(-1)\oplus\oO_{\pp(1,2)}(-2)$.  

Let $u_1, u_2$ be the hyperplane classes of $\pp(1,2)$, and $\hbar$ the first Chern class of the extra $\cc^\times$-action in the $\bbT=(\cc^\times)^2\times\cc^\times$-action on $X_\A$. 
The quantum Batyrev ring is 
$$QH_{\bbT}^*(X_\A)=\frac{\qq[u_1, u_2, \hbar]}{(u_1\star (u_2)^2-Q^{2\ell}(\hbar-u_1)\star(\hbar-u_2)^2)},$$
which is the ring in Theorem \ref{main2:small:quantum:ring}. 

Let $\mathds{1}_{1/2}$ be the identity class of the twisted sector $(X_\A)_{1/2}=B\mu_2$. 
We explain that our formula in Theorem \ref{Quantum:by:a:divisor} generate this class. 
From Theorem \ref{Quantum:by:a:divisor}, the quantum product by divisor is:
$$u_1\star u_2=\frac{Q^{2\ell}}{1-Q^{2\ell}}(\hbar-u_1)\cdot (\hbar-u_2)+ \frac{-Q^\ell}{1-Q^{2\ell}}\mathds{1}_{1/2}\cdot \hbar.$$
Hence we get the class $\mathds{1}_{1/2}$ up to $\hbar$. 
Now we explain that from formulas of quantum product by divisors we can get the relation in the quantum ring. 
Using $u_2$ to do quantum product on the above formula, we get 
$$u_1\star u_2\star u_2=\frac{-Q^{\ell}}{1-Q^{2\ell}}\mathds{1}_{1/2}\cdot \hbar\star u_2,$$
since by Theorem \ref{Quantum:by:a:divisor}, when doing quantum product with $(\hbar-u_1)\cdot (\hbar-u_2)$, the integration of $(\hbar-u_1)\cdot (\hbar-u_2)$ over $\pp(1,2)$ is zero.  Hence we get:
$$u_1\star u_2\star u_2=\frac{-Q^{\ell}}{1-Q^{2\ell}}\cdot \frac{-Q^{3\ell}}{1-Q^{2\ell}}(\hbar-u_1)\cdot (\hbar-u_2)^2.$$
Similarly we have:
$$(\hbar-u_1)\star (\hbar-u_2)=(\hbar-u_1)\cdot(\hbar-u_2)+\frac{Q^{2\ell}}{1-Q^{2\ell}}(\hbar-u_1)\cdot (\hbar-u_2)+ \frac{-Q^\ell}{1-Q^{2\ell}}\mathds{1}_{1/2}\cdot \hbar.$$
Hence
$$(\hbar-u_1)\star (\hbar-u_2)\star (\hbar-u_2)=\frac{-Q^{\ell}}{1-Q^{2\ell}}\cdot \frac{-Q^\ell}{1-Q^{2\ell}}(\hbar-u_1)\cdot (\hbar-u_2)^2.$$
Then the relation is obtained by multiplying $Q^{2\ell}$ on above. 
\end{example}


\end{document}